\documentclass[12pt, twoside, a4paper]{article}
\usepackage[T1]{fontenc}
\usepackage{color}
\usepackage[utf8]{inputenc}
\usepackage[frenchb,english]{babel}
\usepackage{makeidx}
\usepackage{soul}\usepackage{amsmath,caption,fancybox}
\usepackage{lmodern}
\usepackage{ifthen}
\usepackage{color}
\usepackage{epsfig,psfrag}
\usepackage{boiboite}
\usepackage{graphicx,graphics,latexsym,amssymb, pifont, array}
\usepackage[top=2.5cm,bottom=2.5cm,right=2.5cm,left=2.5cm]{geometry}
\usepackage{float, caption}
\makeindex

\font\bb=msbm10
\def\Z{\hbox{\bb Z}}

\def\R{\hbox{\bb R}}

\def\N{\hbox{\bb N}}

\usepackage{amsthm}
\newtheorem{theorem}{Theorem}[section]
\newtheorem{lemma}[theorem]{Lemma}
\newtheorem{proposition}[theorem]{Proposition}

\theoremstyle{theorem}
\newtheorem{definition}{Definition}[section]

\theoremstyle{remark}

\renewcommand\epsilon{\varepsilon}
\newcommand\e{\varepsilon}

\newcommand\om{\underline{\omega}}
\def \g{\underline{g}}
\def \G{\overline{g}}
\def \v{\underline{v}}
\def \V{\overline{v}}
\def \u{\underline{u}}

\everymath{\displaystyle}
\title{Pulsating fronts for bistable on average reaction-diffusion equations in a time periodic environment}
\author{Benjamin CONTRI \footnote{Address correspondence to Benjamin Contri: benjamin.contri@univ-amu.fr}   \\ \\
\small{Aix Marseille Universit\'e, CNRS, Centrale Marseille}\\
\small{Institut de Math\'ematiques de Marseille, UMR 7373, 13453 Marseille, France}}
\date{}

\begin{document}
\maketitle

\begin{abstract}
\noindent This paper is devoted to reaction-diffusion equations with bistable nonlinearities depending periodically on time. These equations admit two linearly stable states. However, the reaction terms may not be bistable at every time. These may well be a periodic combination of standard bistable and monostable nonlinearities. We are interested in a particular class of solutions, namely pulsating fronts. We prove the existence of such solutions in the case of small time periods of the nonlinearity and in the case of small perturbations of a nonlinearity for which we know there exist pulsating fronts. We also study uniqueness, monotonicity and stability of pulsating fronts.
\end{abstract}


\section{Introduction and main results}
\label{introduction}
In this paper we investigate equations of the type
\begin{equation}
\label{equationp}
u_t-u_{xx}=f^T(t,u),~~~~t \in \R,~x\in \R,
\end{equation}
where 
$$f^T(t+T,u)=f^T(t,u),~~~~\forall t \in \R,~~\forall u \in [0,1],$$
and
\begin{equation}
\label{01}
f^T(t,0)=f^T(t,1)=0,~~~~\forall t\in \R.
\end{equation}
Throughout this article, we assume the function $f^T:\R \times [0,1] \to \R$ is of class $\mathcal{C}^{1}$ with respect to $t$ uniformly for $u \in [0,1]$, and $\mathcal{C}^{2}$ with respect to $x$ uniformly for $t \in \R$. The main hypotheses imposed on the function $f^T$ are the following 
\begin{equation}
\label{bismoyT}
\frac{1}{T}\int_0^T f^T_u(s,0)ds<0~~\text{and}~~\frac{1}{T}\int_0^T f^T_u(s,1)ds<0.
\end{equation}
We say the function $f^T$ is bistable on average if it satisfies hypotheses \eqref{01} and \eqref{bismoyT}.
\vspace{0.5em}\\ 
We begin by recalling the definition of monostable and bistable homogeneous nonlinearities. A function $f:[0,1]\to\R$ is said monostable if it satisfies $f(0)=f(1)=0$ and $f>0$ on (0,1). If in addition to this we have $f(u)\leq f'(0)u$ on $(0,1)$, we say that $f$ is of KPP type.   A function $f:[0,1]\to\R$ is called bistable if there exists $\theta \in (0,1)$ such that $f(0)=f(\theta)=f(1)=0$, $f<0$ on $(0,\theta)$ and $f>0$ on $(\theta,1)$.
\vspace{0.5em}\\
We give now two examples of bistable on average functions $f^T:\R \times [0,1] \to \R$. The first example is a bistable homogeneous function balanced by a periodic function depending only on time. Namely, if $g:[0,1]\to \R$  is a bistable function and $m:\R \to \R$ is a $T$-periodic function, then the function $f^T$ defined by $f^T(t,u)=m(t)g(u)$ is bistable on average if and only if the quantity $\textstyle{\frac{1}{T}\int_0^T m(s)ds}$ is positive and both $g'(0)$ and $g'(1)$ are negative. The second example of bistable on average function is a combination of a bistable homogeneous function and a monostable homogeneous function, both balanced by periodic functions (with the same period) depending only on time. Namely, if $g_1:[0,1]\to \R$ is a monostable function, $g_2:[0,1]\to \R$  is a bistable function, and $m_1,m_2:\R \to \R$  are two $T$-periodic functions, then the function $f^T$ defined by $f^T(t,u)=m_1(t)g_1(u)+m_2(t)g_2(u)$ is bistable on average if and only if we have $\mu_1g_1'(0)+\mu_2g_2'(0)<0~\text{and}~\mu_1g_1'(1)+\mu_2g_2'(1)<0$, where $\textstyle{\mu_i:=\frac{1}{T}\int_0^T m_i(s)ds}$. It is important to note that for a bistable on average function, there can very well exist times $t$ for which the homogeneous function $f^T(t,\cdot)$ is not a bistable function in the sense of homogeneous nonlinearities. Indeed, 
if we set  in the previous case $g_1(u)=u(1-u)$, $g_2(u)=u(1-u)(u-\theta)$ with $0<\theta<1$, $m_1(t)=\sin(2\pi t)$ and $m_2(t)=1-\sin(2\pi t)$, we can notice that although the function $f^1(t,u)=m_1(t)g_1(u)+m_2(t)g_2(u)$ is bistable on average, the homogeneous function $f^1(1/4,\cdot)$ is of KPP type.

\subsubsection*{Context}
The study of reaction-diffusion equations began in the 1930's. Fisher \cite{Fisher} and Kolmogorov, Petrovsky and Piskunov \cite{KPP} were interested in the equation
\begin{equation}
\label{eqhomogene}
u_t-u_{xx}=f(u),~~~~t \in \R,~x\in \R,
\end{equation}
with a nonlinearity $f$ of KPP type. They proved existence and uniqueness (up to translation) of planar fronts $U_c$ of speed $c$, for all speeds $\textstyle{c \geq c^*:=2 \sqrt{f'(0)}}$, that is, for all $c\geq c^*$, there exists a function $u_c$ satisfying \eqref{eqhomogene} and which can be written $u_c(t,x)=U_c(x-ct),~~\text{with}~0<U_c<1,~U_c(-\infty)=1~\text{and}~U_c(+\infty)=0.$
Numerous articles have been dedicated to the study of existence, uniqueness, stability, and other properties of planar fronts for various nonlinearities, see e.g. \cite{AW,Fife,Kan,Lau,Sat}. In particular, for bistable nonlinearities, there exists a unique (up to translation) planar front $U(x-ct)$ and a unique speed $c$ solution of \eqref{eqhomogene}. When the nonlinearity is not homogeneous, there are no  planar front solutions of \eqref{equationp} anymore. For equations with coefficients depending on the space variable, Shigesada, Kawasaki and Teramoto \cite{SKT} defined in 1986  a notion more general than the planar fronts, namely the pulsating fronts. This notion can be extended for time dependent equations as follows.
\begin{definition}
A pulsating front connecting $0$ and $1$ for equation \eqref{equationp} is a solution $u:\R\times \R \to [0,1]$ such that there exists a real number $c$ and a function $U:\R \times \R \to [0,1]$ verifying 
\begin{equation*}
\begin{cases}
u(t,x)=U(t,x-ct),~~~~\forall t\in \R,~~\forall x \in \R,\\
U(\cdot,-\infty)=1,~~U(\cdot,+\infty)=0,~~~~\text{uniformly on}~\R,\\
U(t+T,x)=U(t,x),~~~~\forall t\in \R,~~\forall x \in \R.
\end{cases}
\end{equation*}
\end{definition}
\noindent
So, a pulsating front connecting $0$ and $1$ for equation \eqref{equationp} is a solution couple $(c,U(t,\xi))$ of the problem 
\begin{equation}
\label{travel}
\begin{cases}
U_t-cU_{\xi}-U_{\xi \xi}-f^T(t,U)=0,~~~~\forall (t,\xi) \in \R \times \R, \\
U(\cdot,-\infty)=1,~~U(\cdot,+\infty)=0,~~~~\text{uniformly on}~\R,\\
U(t+T,\xi)=U(t,\xi),~~~\forall (t,\xi) \in \R \times \R. \\
\end{cases}
\end{equation}
For an environment depending on space only, we can refer to \cite{BH02,BHN1,BHN2,BHR,DHZ,FZ,HZ,LZ,W02,Xin,Xin1} for some existence, uniqueness and stability results. As far as environments depending on time (and possibly on space), Nolen, Rudd and Xin in \cite{NRX} were interested in equations with a homogenous nonlinearity and an advection coefficient depending periodically on space and on time. Frejacques in \cite{Frejacques} proved the  existence of pulsating fronts in the case of  a time periodic environment with positive and combustion nonlinearities. Nadin in \cite{Nadin09} proved the existence of pulsating fronts in an environment depending on space and time with KPP type nonlinearity. If we consider Nadin's results in the context of our equation, he imposes in his existence results that the steady state $0$ is unstable in the sense that the principal eigenvalue associated with the equilibrium $0$ is negative. Yet, we shall see in section \ref{characterization} that hypotheses \eqref{01} and \eqref{bismoyT} in our paper are equivalent to the fact that $0$ and $1$ are stable steady states, that is, the principal eigenvalues associated with equilibria $0$ and $1$ are positive. Shen in \cite{Shen2} and \cite{Shen1} defined and proved the existence of pulsating fronts in the case of an almost-periodic environment with a bistable nonlinearity, that is, for functions $f$ which satisfy $f(\cdot,0)=f(\cdot,1)=0$ and are negative near the equilibrium $0$ and positive near the equilibrium $1$ for any time. More exactly, it is assumed that there exists $\gamma>0$ and $\delta \in (0,1/2)$ such that
$$\begin{cases}
f(t,u)\leq -\gamma u,~~~~~~~\forall (t,u) \in \R \times [0,\delta],\\
f(t,u)\geq \gamma (1-u),~~\forall (t,u) \in \R \times [\delta,1].
\end{cases}$$  
Alikakos, Bates and Chen in \cite{ABC} (in the case of a periodic nonlinearity) and Shen in \cite{Shen3} (in the case of an almost periodic nonlinearity) consider as in our paper the equation \eqref{equationp}. They impose the  $\text{Poincar\'e}$ map associated with the function $f^T$ has exactly two stable fixed points and one unstable fixed point in between. They prove under this hypothesis there exists a unique pulsating front $(U,c)$ solution of the problem  \eqref{travel}. They show that for each $t$, the function $U(t,\cdot)$ is monotonic and that $U$, $U_\xi$ and $U_{\xi\xi}$ exponentially approach their limits as $\xi \to \pm\infty$. They also prove a global exponential stability result. In section \ref{caracterisation}, we will see that hypotheses $\eqref{01}$ and \eqref{bismoyT} are equivalent to the fact that $0$ and $1$ are two stable fixed points of the $\text{Poincar\'e}$ map associated with $f^T$. Let us note that in this paper, we do not impose the uniqueness of intermediate fixed points between $0$ and $1$. 
\vspace{0.5em}\\
We now give the main results of the paper.
\subsubsection*{Uniqueness and monotonicity of pulsating fronts}
In this paper, we begin by showing the monotonicity of pulsating front solving \eqref{travel}. Then we use this result to prove the uniqueness (up to translation) of the pair $(U,c)$ solving \eqref{travel}. For that purpose, we use some new comparison principles adapted to hypotheses \eqref{01} and \eqref{bismoyT}, that we show using sliding methods. We have the following theorem.
\begin{theorem}
\label{existenceunicite}
There exists at most one couple $(c,U)$ solution of problem \eqref{travel}, the function $U(t,\xi)$ being unique up to shifts in $\xi$. In this case, we have that 
$$U_\xi(t,\xi)<0,~~~~\forall t \in \R,~~\forall \xi \in \R.$$
\end{theorem}
\subsubsection*{Asymptotic stability of pulsating fronts}
We now investigate global stability of pulsating fronts. In case of homogeneous bistable nonlinearities, Fife and McLeod in \cite{Fife} proved the global stability of planar fronts. For heterogeneous bistable nonlinearities, Alikakos, Bates and Chen \cite{ABC} (in case depending on time variable) and Ding, Hamel and Zhao \cite{DHZ} (in case depending on space variable) also proved a global exponential stability result.
In our case with assumptions \eqref{01} and \eqref{bismoyT}, the following stability result is proved.
\begin{theorem}
\label{stability}
Assume there exists a pulsating front $U$ with speed $c$ solution of \eqref{travel}. We consider a solution $u$ of the Cauchy problem 
$$\begin{cases}
u_t-u_{xx}=f^T(t,u)~~~~\forall t>0,~~ \forall x \in \R, \\
u(0,x)=h(x),~~~~\hspace{1.35em}~\forall x \in \R,
\end{cases}$$
where the initial condition $h:\R \to [0,1]$  is uniformly continuous. We denote $v(t,\xi)=u(t,x)=u(t,\xi+ct)$. There exists a constant $\gamma \in (0,1)$ depending only on $f^T$ such that if 
$$\liminf\limits_{\xi \to -\infty}h(\xi)>1-\gamma~~\text{and}~~\limsup\limits_{\xi \to +\infty}h(\xi)<\gamma,$$
then, there exists $\xi_0 \in \R$ such that 
$$\lim\limits_{t\to +\infty} \big( v(t,\xi)-U(t,\xi+\xi_0) \big)=0,~~~~\text{uniformly for}~\xi \in \R.$$
\end{theorem}
\noindent
Roughly speaking, if the initial condition $h$ "looks like" a front, then $v$ converges to a front as $t \to +\infty$. To prove this theorem we use the method of sub- and supersolution. We adapt here some ideas used in \cite{Fife} in case of a bistable homogeneous nonlinearity to our equation \eqref{equationp} with assumptions \eqref{01} and \eqref{bismoyT}. 
\subsubsection*{Existence and convergence of pulsating fronts for small periods}
We are interested here in understanding the role of the period $T$ of the function $f^T$ in the limit of small periods. We consider nonlinearities of the form 
$$f^T(t,u)=f(\frac{t}{T},u),~~~~\forall t \in \R,~~\forall u \in [0,1].$$
The function $f$ is $1$-periodic in time, and hypothesis \eqref{bismoyT} becomes 
$$\int_0^1 f_u(s,0)ds<0~~\text{and}~~\int_0^1 f_u(s,1)ds<0.$$
Consequently, the sign of the quantities \eqref{bismoyT} do not depend on the period $T$. In order to understand the homogeneization limit as $T \to 0^+$, we define the averaged nonlinearity
\begin{equation}
\label{fonctiong}
\begin{array}{ccccc}
g & : & [0,1] & \to & \R \\
 & & u & \mapsto & \int_0^1 f(s,u)ds \\
\end{array}
\end{equation}
We assume that the function $g$ is a bistable function, that is, there exists $\theta_g \in (0,1)$ such that 
\begin{equation*}
\label{bistable}
\begin{cases}
g(0)=g(\theta_g)=g(1)=0, \\
g(u)<0,~~\forall u \in (0,\theta_g),~~~~g(u)<0,~~\forall u \in (\theta_g,1).\\
\end{cases}
\end{equation*}
We also assume that 
\begin{equation}
\label{sdgtz}
g'(\theta_g)>0.
\end{equation}
Let us noticing that according to \eqref{bismoyT}, one gets $g'(0)<0$ and $g'(1)<0$. We have the following existence theorem.
\begin{theorem}
\label{petiteex}
Under the above assumptions, there exists $T_f>0$ such that for all $T\in(0,T_f)$,  there exists a unique pulsating front $(U_T(t,\xi),c_T)$ solving 
$$\begin{cases}
(U_T)_t-c_T(U_T)_{\xi}-(U_T)_{\xi \xi}=f^{T}(t,U_T),~~~~\text{on}~\R \times \R, \\
U_T(\cdot,-\infty)=1,~~U_T(\cdot,+\infty)=0,~~~~\text{uniformly on}~\R,\\
U_T(t+T,\xi)=U_T(t,\xi),~~~~\forall (t,\xi) \in \R \times \R, \\
U_T(0,0)=\theta_g.
\end{cases}$$
\end{theorem}
\noindent
We are then interested in the convergence of the couple $(c_T,U_T)$ as $T \to 0$. We recall from \cite{AW} that for the bistable nonlinearity $g$, there exists a unique planar fronts $(c_g,U_g)$ solving
\begin{equation*}
\begin{cases}
U_g''+c_gU_g'+g(U_g)=0,~~\text{on}~\R,\\
U_g(-\infty)=1,~~U_g(+\infty)=0,\\
U_g(0)=\theta_g.
\end{cases}
\end{equation*}
For any $1 \leq p \leq +\infty$, we define 
$$W^{1,2;p}_{\text{loc}}(\R^2)=\left\{U\in L^p_{\text{loc}}(\R^2)~|~\partial_tU,~\partial_\xi U,~\partial_{\xi \xi}U \in L^p_{\text{loc}}(\R^2)\right\}.$$
For any $k \in \N$ and any $\alpha \in (0,1)$, we also define  $\mathcal{C}^{k,\alpha}(\R^2)$ the space of functions of class $\mathcal{C}^{k}(\R^2)$ with the $k^{th}$ partial derivatives $\alpha$-H\"older.
\vspace{0.5em}\\
In terms of convergence, we extend the function $U_g$ on $\R^2$ by $U_g(t,\xi)=U_g(\xi)$ for any $(t,\xi) \in \R^2$, and we have the following result 
\begin{theorem}
\label{petiteconv}
As $T\to0$, the speed $c_T$ converges to $c_g$ and $U_T$ converges to $U_T$ in $W^{1,2;p}_{\text{loc}}(\R^2)$ weakly and in $\mathcal{C}^{0,\alpha}_{\text{loc}}(\R^2)$, for any $1< p < +\infty$ and for any $\alpha \in (0,1)$.
\end{theorem}
\noindent
Let us mention that for nonlinearities depending only on space variable, such convergence results were proved by Ding, Hamel and Zhao \cite{DHZ} in the bistable case and by El Smaily \cite{smaily} in the KPP case.
\subsubsection*{Existence and convergence of pulsating fronts for small perturbations}
In this part, we consider some families of functions $f^{T,\e}: \R\times [0,1] \to \R$ having the same regularity as $f^T$ and  such that 
\begin{equation}
\label{benjamin}
\begin{cases}
f^{T,\e}(t,0)=f^{T,\e}(t,1)=0,~~\forall t\in \R,\\
f^{T,\e}(t,u)=f^{T,\e}(t+T,u),~\hspace{-0.1em}~\forall (t,u) \in\R \times [0,1].
\end{cases}
\end{equation}
We also suppose that there exists a bounded function $\omega(\e):(0,+\infty) \to \R$ satisfying $\omega(\e) \xrightarrow{\e \to 0}0$ and such that
\begin{equation}
\label{perturbation2}
|f^T_u(t,u)-f_u^{T,\e}(t,u)|\leq\omega(\e),~~\forall (t,u) \in [0,T] \times [0,1].
\end{equation}
We will first show that if $f^T$ satisfies the hypotheses of existence and uniqueness theorem of Alikakos, Bates and Chen \cite{ABC}, then for $\e>0$ small enough, the $\text{Poincar\'e}$ map associated with $f^{T,\e}$ also verifies it. As consequence, the following theorem holds.
\begin{theorem}
\label{pertex}
We suppose the $\text{Poincar\'e}$ map associated to $f^T$ has exactly two stable fixed points $0$ and $1$ and one unstable  fixed point $\alpha_0$ between both.
\vspace{0.5em}\\
Then, there exists $\e_0>0$ such that for all $\e\in(0,\e_0)$, there exists a unique pulsating front $(U_\e(t,\xi),c_\e)$ solving 
\begin{equation}
\label{travelepsilon}
\begin{cases}
(U_\e)_t-c_\e(U_\e)_{\xi}-(U_\e)_{\xi \xi}-f^{T,\e}(t,U_\e)=0,~~~~\text{on}~ \in\R \times \R, \\
U_\e(\cdot,-\infty)=1,~U_\e(\cdot,+\infty)=0,~~~\text{uniformly on}~\R,\\
U_\e(t+T,\xi)=U_\e(t,\xi),~~~~\forall (t,\xi) \in \R \times \R, \\
U_\e(0,0)=\alpha_0.
\end{cases}
\end{equation}
\end{theorem}
\noindent
Let us note that by hypothesis on $f^T$, there exists a unique pulsating front $(U_T,c_T)$ solving problem \eqref{travel} with $U_T(0,0)=\alpha_0$. We have then the following convergence result 
\begin{theorem}
\label{pertconv}
As $\e \to0$, the speed $c_\e$ converges to $c_T$ and $U_\e$ converges to $U_T$ in $W^{1,2;p}_{\text{loc}}(\R^2)$ weakly and in $\mathcal{C}^{0,\alpha}_{\text{loc}}(\R^2)$, for any $1 < p < +\infty$ and for any $\alpha \in (0,1)$.
\end{theorem} 
\subsubsection*{Outline}
Section \ref{characterization} of this paper is devoted to some equivalent formulations of hypotheses \eqref{01} and \eqref{bismoyT}, in particular using $\text{Poincar\'e}$ map and principal eigenvalue. In the following two sections, we prove uniqueness, monotonicity and uniform stability of pulsating front solution of \eqref{equationp} that is, Theorems \ref{existenceunicite} and  \ref{stability}. In Section $5$, we prove Theorems \ref{petiteex} and \ref{petiteconv} on the homogenization limit. In Section $6$, we prove Theorems \ref{pertex} and \ref{pertconv} on a small perturbation of a given pulsating front.
\section{Preliminaries on the characterization of the asymptotic stability of equilibrium state}
\label{characterization}
This part is devoted to the study of various characterizations of bistable on average functions. As we mentioned it previously, it is necessary to know its various points of view to be able to place the results of our paper in the literature already existing. We begin by defining the notion of equilibrium state.
\begin{definition}
Consider the problem
\begin{equation}
\label{equil}
\begin{cases} \partial_tU-c\partial_\xi U-\partial_{\xi \xi}U=f^T(t,U)~~\text{on}~\R^2, \\
 U(t,\xi)=U(t+T,\xi),~~\forall t \in \R, ~\forall \xi \in \R. 
 \end{cases}
 \end{equation}
The $T$-periodic solutions of the equation $y'=f^T(t,y)$ are called equilibrium states of \eqref{equil}. 
\end{definition}
\noindent
Indeed, if $U:\R^2\to[0,1]$ is a solution of \eqref{equil} and if there exists a function $t \mapsto \theta(t)$ such that 
$|U(t,\xi)-\theta(t)| \xrightarrow{\xi \to \infty} 0$ for all $t\in \R,$ then, by standard parabolic estimates, the function $\theta$ is a $T$-periodic solution of the equation $y'=f^T(t,y)$ on $\R$. We recall the notion of uniformly asymptotic stability of such an equilibrium state.
\begin{definition}
\label{defstab}
Let $t \mapsto \theta(t)$ be an equilibrium state of \eqref{equil}. Let $(t,\xi) \mapsto U(t,\xi)$ be another solution of the same equation. We say that $\theta$ is a uniformly asymptotic stable equilibrium  if there exists $\eta>0$ such that 
\begin{center}
$\big( \forall \xi \in \R,~|U(0,\xi)-\theta(0)|<\eta \big)$ $\Longrightarrow$ $\lim\limits_{t\to+\infty}|U(t,\cdot)-\theta(t)|=0$ unif. on $\R$. 
\end{center}
If $\theta$ is not a uniformly asymptotic stable equilibrium, we say it is a uniformly asymptotic unstable equilibrium.\end{definition}
\noindent
We approach now the notion of principal eigenvalue.
\begin{proposition}{\cite{Hess},\cite{NadinVP}}
Let $t \mapsto \theta(t)$ an equilibrium state of \eqref{equil}. There exists a constant $\lambda_{\theta,f^T}$ and a function $\Phi_{\theta,f^T} \in \mathcal{C}^{1,2}(\R\times \R,\R)$ such that 
\begin{equation}
\label{vpp}
\begin{cases}
\partial_t \Phi_{\theta,f^T} -c\partial_\xi \Phi_{\theta,f^T} -\partial_{\xi\xi} \Phi_{\theta,f^T} = f^T_u(t,\theta(t))\Phi_{\theta,f^T} + \lambda_{\theta,f^T} \Phi_{\theta,f^T}~~\text{on}~\R^2, \\
\Phi_{\theta,f^T}>0~~\text{on}~\R^2, \\
\Phi_{\theta,f^T}(\cdot,\xi)~\text{is}~T-\text{periodic},~\forall \xi \in \R,\\
\Phi_{\theta,f^T}(t,\cdot)~\text{is}~1-\text{periodic},~~\forall t \in \R.
\end{cases}
\end{equation}
The real number $\lambda_{\theta,f^T}$ is unique. It is called the principal eigenvalue associated with the function $f^T$ and the equilibrium state $\theta$.
The function $\Phi_{\theta,f^T}$ is unique up to multiplication by a positive constant. It is called the principal eigenfunction associated with the function $f^T$ and the equilibrium state $\theta$.
\end{proposition}
\noindent
We give an explicit formulation of the principal eigenvalue.
\begin{proposition}
\label{enplus}
The function $\Phi_{\theta,f^T}$ does not depend on the variable $\xi$, and the constant $\lambda_{\theta,f^T}$ is given by 
$$\lambda_{\theta,f^T}=-\frac{1}{T} \int_0^T f^T_u(s,\theta(s)) ds.$$
\end{proposition}
\begin{proof}
We know that $\Phi_{\theta,f^T}$ is unique up to multiplication by a positive constant. Let us suppose for example that $\|\Phi_{\theta,f^T}\|_\infty=1$. Let $\xi_0 \in \R$. The function $\Phi_{\theta,f^T}(t,\xi+\xi_0)$ is also a positive solution of the problem \eqref{vpp}. Yet $\|\Phi_{\theta,f^T}(\cdot,\cdot+\xi_0)\|_\infty=1$. So, by uniqueness $\Phi_{\theta,f^T}(t,\xi+\xi_0)=\Phi_{\theta,f^T}(t,\xi)$ for any $(t,\xi)$ in $\R^2$. Since $\xi_0$ is arbitrary, it follows that $\Phi_{\theta,f^T}$ does not depend on the variable $\xi$. Furthermore, the first equation in \eqref{vpp} becomes 
$$\partial_t \Phi_{\theta,f^T} = f^T_u(t,\theta(t))\Phi_{\theta,f^T} + \lambda_{\theta,f^T} \Phi_{\theta,f^T},~~\forall t \in \R.$$
We divide this equation by $\Phi_{\theta,f^T}$, then we integrate between $0$ and $T$. According to the fact that $\Phi_{\theta,f^T}$ is a $T$-periodic function, we obtain the expression of $\lambda_{\theta,f^T}$ given in Proposition \ref{enplus}.
\end{proof}
\noindent
We give a characterization of the uniformly asymptotic stability of the equilibrium state $\theta$ from the principal eigenvalue $\lambda_{\theta,f^T}$. 
\begin{proposition}
\label{equiv}
If $\lambda_{\theta,f^T}>0$ (resp. $<0$), then the equilibrium state $\theta$ is uniformly asymptotically stable (resp. unstable).
\end{proposition}
\begin{proof}
We are going to handle the case where $\lambda_{\theta,f^T}>0$. We consider a solution $U(t,\xi)$ of \eqref{equil}. We saw that the function $\Phi_{\theta,f^T}=\Phi_{\theta,f^T}(t)$ satisfies \\
$$\begin{cases}
\partial_t \Phi_{\theta,f^T} = f^T_u(t,\theta(t))\Phi _{\theta,f^T}+ \lambda_{\theta,f^T} \Phi_{\theta,f^T}~~\text{on}~\R, \\
\Phi_{\theta,f^T}>0, \\
\Phi_{\theta,f^T}~\text{is}~T-\text{periodic}.\\
\end{cases}$$
There exists $\e_0>0$ small enough such that for any $t\geq 0$ we have
\begin{multline*}
 -\frac{\lambda_{\theta,f^T}}{4}\e_0e^{-\frac{\lambda_{\theta,f^T}}{2} t}\Phi_{\theta,f^T}(t)
  \leq f^T(t,\theta(t)+\e_0e^{-\frac{\lambda_{\theta,f^T}}{2} t}\Phi_{\theta,f^T}(t))-
  \\ f^T(t,\theta(t))-f^T_u(t,\theta(t))\e_0e^{-\frac{\lambda_{\theta,f^T}}{2} t}\Phi_{\theta,f^T}(t) 
 \leq \frac{\lambda_{\theta,f^T}}{4}\e_0e^{-\frac{\lambda_{\theta,f^T}}{2} t}\Phi_{\theta,f^T}(t).
 \end{multline*}
We note $\chi(t)=\theta(t)+\e_0e^{-\frac{\lambda_{\theta,f^T}}{2} t}\Phi_{\theta,f^T}(t)$. For any $t \geq 0$, we have 
$$\chi_t(t) -c\chi_{\xi}(t)-\chi_{\xi \xi}(t)\\ -f^T(t,\chi(t)) \geq \e_0e^{-\frac{\lambda_{\theta,f^T}}{2} t}\Phi_{\theta,f^T}(t)[\lambda_{\theta,f^T}-\frac{\lambda_{\theta,f^T}}{2}-\frac{\lambda_{\theta,f^T}}{4}]>0.$$
If we suppose that for any $\xi$ in $\R$, we have $U(0,\xi)\leq \theta(0)+\e_0\Phi_{0,f^T}(0)$, then, applying the maximum principle, we have that 
$$U(t,\xi)\leq \theta(t)+\e_0e^{-\frac{\lambda_{\theta,f^T}}{2} t}\Phi_{\theta,f^T}(t),~~\forall t \geq 0,~~\forall \xi \in \R.$$
In the same way, possibly reducing $\e_0$, one can show that if we suppose for any $\xi \in \R$ that we have $ \theta(0)-\e_0\Phi_{0,f^T}(0) \leq U(0,\xi)$, then 
$$\theta(t)-\e_0e^{-\frac{\lambda_{\theta,f^T}}{2} t}\Phi_{\theta,f^T}(t)\leq U(t,\xi),~~\forall t \geq 0,~~\forall \xi \in \R.$$
Consequently, $\lim\limits_{t\to+\infty}|U(t,\xi)-\theta(t)|=0$ uniformly on $\R$.
\vspace{0.5em}\\
We are now interested in case where $\lambda_{\theta,f^T}<0$. Let $\eta>0$. There exists $\e_{0,T}>0$ small enough such that $\textstyle{\e_{0,T}\leq  \frac{\eta}{\Phi_{\theta,f^T}(0)}}$ and,  for any $t \in [0,T]$ and any $\e \in (0,\e_{0,T})$ we have 
\begin{multline*}
 \frac{\lambda_{\theta,f^T}}{4}\e e^{-\frac{\lambda_{\theta,f^T}}{2} t}\Phi_{\theta,f^T}(t)
  \leq f^T(t,\theta(t)+\e e^{-\frac{\lambda_{\theta,f^T}}{2} t}\Phi_{\theta,f^T}(t))-
  \\ f^T(t,\theta(t))-f^T_u(t,\theta(t))\e e^{-\frac{\lambda_{\theta,f^T}}{2} t}\Phi_{\theta,f^T}(t) 
 \leq -\frac{\lambda_{\theta,f^T}}{4}\e e^{-\frac{\lambda_{\theta,f^T}}{2} t}\Phi_{\theta,f^T}(t).
 \end{multline*}
We note $\chi_\e(t)=\theta(t)+\e e^{-\frac{\lambda_{\theta,f^T}}{2} t}\Phi_{\theta,f^T}(t)$. The same calculations as previously give us that for any $t \in [0,T]$ we have
$$(\chi_\e)_t(t) -c(\chi_\e)_{\xi}(t)-(\chi_\e)_{\xi \xi}(t) -f^T(t,\chi_\e(t)) \leq 0.$$
We define $U_\e:\R^+\times \R \to \R$ solution of the Cauchy problem
$$\begin{cases}
(U_\e)_t -c(U_\e)_{\xi}-(U_\e)_{\xi \xi} -f^T(t,U_\e)~~\text{on}~\R^+ \times \R, \\
U_\e(0,\cdot)=\chi_\e(0)~~\text{on}~\R.
\end{cases}$$
Applying the maximum principle on $[0,T] \times \R$ to the function $U_\e$ and $\chi_\e$, we obtain that  $U_\e \geq \chi_\e$ $\text{on}~[0,T]\times \R.$ In particular, since $\chi_\e$ is a nondecreasing function on $[0,T]$, we have 
\begin{equation}
\label{1205}
U_\e(T,\cdot)\geq \chi_\e(0)~~\text{on}~\R.
\end{equation}
The function $U_\e(\cdot+T,\cdot)$ satisfies the same equation as $U_\e$. According to \eqref{1205}, we can apply the maximum principle  on $[0,T] \times \R$ to the function $U_\e(\cdot+T,\cdot)$ and $\chi_\e$. We obtain for $t=T$ that
\begin{equation*}
U_\e(2T,\cdot)\geq \chi_\e(T) \geq \chi_\e(0)~~\text{on}~\R.
\end{equation*}
For any $n \in \N^*$, we can show by induction that
\begin{equation*}
U_\e(nT,\cdot)\geq \chi_\e(0)~~\text{on}~\R.
\end{equation*}
According to the fact that $\chi_\e(0) > \theta(0),$ and that $\theta(nT)=\theta(0)$ we have that $\lim\limits_{t \to +\infty} (U_\e(t,\cdot)-\theta(t) )\neq 0$ for any $\xi \in \R$, although we have $|U_\e(0,\xi)-\theta(0)|\leq \eta$, for any $\xi \in \R$
\end{proof}
\noindent
Finally, we connect now the previous notions with the notion of $\text{Poincar\'e}$ map associated with the function $f^T$. \begin{definition}
\label{defPoin}
For any $\alpha \in [0,1]$, let $w(\alpha,\cdot)$ be the solution of the Cauchy problem 
\begin{equation} 
\label{diff}
\begin{cases}
y'=f^T(t,y), \\
y(0)=\alpha.
\end{cases}
\end{equation} 
The $\text{Poincar\'e}$ map associated with $f^T$ is the function $P:[0,1] \to [0,1]$ such that 
$$P(\alpha):=w(\alpha,T).$$
Let $\alpha_T$ be a fixed point of $P$.
We say that $\alpha_T$ is stable (resp. unstable) if $P'(\alpha_T)<1$ (resp. $P'(\alpha_T)>1$).
\end{definition}
\noindent
We give the link between equilibrium states of \eqref{equil} and fixed points of  the $\text{Poincar\'e}$ map associated with $f^T$. First of all, it follows from the definition of $P$ that a real number $\alpha \in [0,1]$ is a fixed point of $P$ if and only if $w(\alpha,\cdot)$ is an equilibrium state of \eqref{equil}.
\begin{proposition}
Let $\alpha$ be a fixed point of $P$. We have 
$$P'(\alpha)=e^{-T{\lambda_{\omega(\alpha,\cdot),f^T}}}.$$
\end{proposition}
\begin{proof}
We have $$\partial_t w(\alpha,t)= f^T(t,w(\alpha,t)),~~\forall t \in \R.$$
Differentiating with respect to $\alpha$, we obtain that $\partial_\alpha w(\alpha,\cdot)$ solves the linear ODE $y'=yf^T_u(s,w(\alpha,s))$. It follows that $$\partial_\alpha w(\alpha,t)=\partial_\alpha w(\alpha,0) e^{\int_0^tf^T_u(s,w(\alpha,s))ds},~~\forall t \in \R.$$
If we take $t=T$, we have that $\partial_\alpha w(\alpha,T)=\partial_\alpha w(\alpha,0) e^{\int_0^Tf^T_u(s,w(\alpha,s))ds}$, and since $\partial_\alpha w(\alpha,0)=1$, we infer that $\partial_\alpha w(\alpha,T)=e^{\int_0^Tf^T_u(s,w(\alpha,s))ds}$. In other words, by Proposition \ref{enplus}, $P'(\alpha)=e^{-T{\lambda_{\omega(\alpha,\cdot),f^T}}}.$
\end{proof}
\noindent
Consequently, the fact that a fixed point $\alpha$ of the $\text{Poincar\'e}$ map associated with $f^T$ is stable (resp. unstable) in the sense of Definition \ref{defPoin} is equivalent to the fact that the principal eigenvalue associated with $w(\alpha,\cdot)$ and $f^T$ is positive (resp. negative), that is, by Proposition \eqref{equiv}, the solution $w(\alpha,\cdot)$ of \eqref{diff} is a uniformly asymptotic stable (resp. unstable) equilibrium of \eqref{equil}.
\vspace{0.5em}\\
In particular, in our paper, the hypothetis \eqref{01} implies that $0$ and $1$ are two fixed points of the $\text{Poincar\'e}$ map associated with $f^T$, and the condition \eqref{bismoyT} is a condition of positivity of the principal eigenvalues associated with $0$ and $1$. In this way, the equilibria $0$ and $1$ are uniformly asymptotically stable for the equation \eqref{equationp}.
\section{Uniqueness and monotonicity of pulsating front}
This section is devoted to the proof of Theorem \ref{existenceunicite}.
\label{caracterisation}
\subsection{Two comparison principles}
\begin{lemma}
\label{TH1comparaison}
Let us fix $c \in \R$, $R_+\in \R$ and $\alpha \in (0,1)$.
We consider two functions $\g$ and $\G$ of class $\mathcal{C}^1(\R \times [0,1],\R)$, T-periodic and such that  
\begin{equation}
\label{compf}
\g(t,u)\leq \G(t,u),~~~\forall t\in \R,~~\forall u \in [0,1].
\end{equation}
We assume $\G$ satisfies  the hypothesis \eqref{01} and the first inequality of \eqref{bismoyT}.
\vspace{0.5em}\\
Suppose there exist two functions $\V:\R \times [R_+,+\infty) \to [0,1],~(t,\xi) \mapsto \V(t,\xi)$ and $\v:\R \times [R_+,+\infty) \to [0,1],~(t,\xi) \mapsto \v(t,\xi)$ of class $\mathcal{C}^{1,\frac{\alpha}{2}}(\R)$ in $t$ uniformly for $\xi \in [R_+,+\infty)$ and of class $\mathcal{C}^{2,\alpha}([R_+,+\infty))$ in $\xi$ uniformly for $t \in \R,$ and such that 
\begin{equation}
\label{usursol}
\partial_t \V-c\partial_\xi \V-\partial_{\xi \xi} \V \geq \G(t,\V)~~\text{on}~\R \times [R_+,+\infty),
\end{equation}
\vspace{-1.5em}
\begin{equation}
\label{usoussol}
\partial_t \underline{v}-c\partial_\xi \underline{v}-\partial_{\xi \xi} \underline{v} \leq \underline{g}(t,\underline{v})~~\text{on}~\R \times [R_+,+\infty),
\end{equation}
\begin{equation}
\label{comp21}
\v(t,R_+) \leq \V(t,R_+),~\forall t \in \R,
\end{equation}
\begin{equation}
\label{limunifu}
\v(\cdot,+\infty)=0~~\text{uniformly on}~\R.
\end{equation}
There exists $\delta_+ \in (0,1)$ depending only on $\overline{g}$ such that if we have 
\begin{equation}
\label{comp23}
\v(t,\xi)\leq \delta_+, ~~ \forall t \in \R,~~\forall \xi \in [R_+,+\infty),
\end{equation}
then
\begin{equation*}
\label{thcomparaison}
\v(t,x)\leq \V(t,x),~~~~\forall t\in \R,~~\forall \xi \in [R_+,+\infty).
\end{equation*}
\end{lemma}
\noindent
\begin{proof}
Let us first introduce a few notations. We denote $\lambda_{0,\G}$ and $\Phi_{0,\G}$ the principal eigenvalue and the principal eigenfunction associated with the function $\G$ and the equilibrium $0$. We saw in Section \ref{characterization} that $\Phi_{0,\G}$ depends only on $t$ (and not on $\xi$). Consequently, we have 
\begin{equation}
\label{vp0}
\Phi_{0,\G}'(t)=(\lambda_{0,\G}+\G_u(t,0))\Phi_{0,\G}(t),~~\forall t \in \R, 
\end{equation}
We also saw in Section \ref{characterization} that $\lambda_{0,\G}>0$. Since $\G$ is of class $\mathcal{C}^1(\R \times [0,1],\R)$ and periodic in $t$, there exists $\delta_+>0$ such that 
\begin{equation}
\label{regf}
\forall t \in \R,~~\forall (u,u') \in [0,1],~~|u-u'|\leq \delta_+ \Rightarrow |\G_u(t,u)-\G_u(t,u')|<\frac{\lambda_{0,\G}}{2}.
\end{equation}
The general strategy to prove Lemma \ref{TH1comparaison} consists in using a sliding method. To do so, we define 
$$\e^*=\inf \big\{ \e\geq 0~|~\frac{\V-\v}{\Phi_{0,\G}}+\e\geq0~\text{on}~\R\times[R_+,+\infty) \big\}.$$ 
We note that $\e^*$ is a real number since $\v$ and $\V$ are bounded and $\textstyle{\min_{\R}{\Phi_{0,\G}}>0}$. Furthermore, by continuity 
\begin{equation}
\label{demcomp1}
\frac{\V-\v}{\Phi_{0,\G}}+\e^*\geq0~\text{on}~\R\times[R_+,+\infty).\\
\end{equation}
We are going to show bwoc that $\e^*=0$. Thus let us suppose that $\e^*>0$. We consider a sequence $(\e_n)_n$ satisfying $\e_n \xrightarrow{n\to +\infty} \e^*$ and $0<\e_n<\e^*$. There exists $(t_n,\xi_n) \in \R \times[R_+,+\infty)$ such that 
\begin{equation}
\label{demcomp2}
\frac{\V(t_n,\xi_n)-\v(t_n,\xi_n)}{\Phi_{0,\G}(t_n)}+\e_n<0.\\
\end{equation}
We write $t_n=k_nT+t_n'$, with $k_n \in \Z$ and $|t_n'|<T$. Since $(t_n')_n$ is bounded, we thus have up to extraction of a subsequence that $t_n' \xrightarrow{n\to +\infty} t^*$. Furthermore, the sequence $(\xi_n)_n$ is also bounded. Indeed, let us suppose it is not the case. So, according to \eqref{limunifu}, there exists $R\in \R$ such that 
$$\forall n \in \N,~~\forall \xi \geq R,~~0\leq \v(t_n,\xi)\leq \frac{\e^*}{4}\min_{\R}\Phi_{0,\G}.$$
Now, if $(\xi_n)_n$ is not bounded, there exists $N \in \N$ such that $\xi_N\geq R$ and $\textstyle{\e_N>\frac{\e^*}{2}}$. So, 
$$\v(t_N,\xi_N)-\e_N\Phi_{0,\G}(t_N)\leq \v(t_N,\xi_N)-\frac{\e^*}{2}\min_{\R} \Phi_{0,\G}\leq \frac{\e^*}{4}\min_{\R} \Phi_{0,\G}-\frac{\e^*}{2}\min_{\R} \Phi_{0,\G}<0.$$
Now, by \eqref{demcomp2} we have $\V(t_N,\xi_N)\leq \v(t_N,\xi_N)-\e_N\Phi_{0,\G}(t_N)$. Hence $\V(t_N,\xi_N)<0$, which contradicts the fact that $\V \in [0,1]$. We thus have up to extraction of a subsequence that $\xi_n \xrightarrow{n\to +\infty} \xi^*\in [R_+,+\infty)$. We define $\v_n(t,\xi)=\v(t+k_nT,\xi),~\text{and}~\V_n(t,\xi)=\V(t+k_nT,\xi).$ As $\g$ and $\G$ are $T-$periodic, $\v_n$ and $\V_n$ satisfy respectively \eqref{usursol} and \eqref{usoussol}. Furthermore, $\v_n$ and $\V_n$ converge up to extraction of a subsequence respectively to $\v^*\in [0,1]$ and $\V^*\in [0,1]$ in $\mathcal{C}^{1,2}_\text{loc}(\R \times [R_+,+\infty))$. Passing to the limit in \eqref{demcomp1} we obtain
\begin{equation*}
\frac{\V^*-\v^*}{\Phi_{0,\G}}+\e^*\geq0~\text{on}~\R\times[R_+,+\infty).\\
\end{equation*}
According to \eqref{demcomp2}, we have 
$$\frac{\V_n(t_n',\xi_n)-\v_n(t_n',\xi_n)}{\Phi_{0,\G}(t_n')}+\e_n < 0.$$
So, passing to the limit, we have 
\begin{equation}
\label{demcomp3}
\frac{\V^*(t^*,\xi^*)-\v^*(t^*,\xi^*)}{\Phi_{0,\G}(t^*)}+\e^*= 0.\\
\end{equation}
We define the open set $$\Omega=\big\{(t,\xi)\in\R\times[R_+,+\infty)~|~0\leq \V^*(t,\xi)<\v^*(t,\xi) \big\}.$$
This set is open by continuity and by \eqref{comp21}. We note that $(t^*,\xi^*)\in \Omega$ because $\V^*(t^*,\xi^*)-\v^*(t^*,\xi^*)=-\e^*\Phi_{0,\G}(t^*)<0$. Furthermore, according to \eqref{comp21}, we have $\V^*(t^*,R_+)-\v^*(t^*,R_+) \geq 0$. So, $\xi^*>R_+$. We will apply a strong maximum principle to the nonnegative function 
$$z=\frac{\V^*-\v^*}{\Phi_{0,\G}}+\e^*.$$ 
There exists $\theta:\R\times [R_+,+\infty) \to [0,1],~(t,\xi) \mapsto \theta(t,\xi)$, with $\theta(t,\xi)$ between $\V^*(t,\xi)$ and $\v^*(t,\xi)$ such that we have on $\R\times [R_+,+\infty)$
\begin{flalign*}
\partial_tz-c\partial_\xi z-\partial_{\xi\xi}z  & \geq \frac{\G(t,\V^*)-\underline{g}(t,\v^*)}{\Phi_{0,\G}} - \frac{\partial_t\Phi_{0,\G}(\V^*-\v^*)}{\Phi_{0,\G}^2} ~~~(\eqref{usursol} ~\text{and}~ \eqref{usoussol})~&&\\
& \geq \frac{\G(t,\V^*)-\G(t,\v^*)}{\Phi_{0,\G}}-\frac{(\G_u(t,0)+\lambda_{0,\G})(\V^*-\v^*)}{\Phi_{0,\G}}~~~\text{(\eqref{vp0} and \eqref{compf})}~&&\\
&= \frac{1}{\Phi_{0,\G}}(\G_u(t,\theta)-\G_u(t,0)-\lambda_{0,\G}) (\V^*-\v^*),&&\\
&= ( \G_u(t,\theta)-\G_u(t,0)-\lambda_{0,\G}(z-\e^*) \G_u(t,\theta)-\G_u(t,0)-\lambda_{0,\G}).
\end{flalign*}
We define the bounded function $\alpha(t,\xi):=\G_u(t,\theta(t,\xi))-\G_u(t,0)-\lambda_{0,\G}$ on $\R \times [R_+,+\infty)$. According to \eqref{comp23} we have $\v^* \leq \delta_+$ on $\R \times [R_+,+\infty)$. Furthermore, we have $0\leq\V^* < \v^*$ on $\Omega$, whence $|\theta|\leq\delta_+$ on $\Omega$. So, by \eqref{regf} we have on $\Omega$
\begin{equation*}
\partial_tz-c\partial_\xi z-\partial_{\xi\xi}z -\alpha(t,\xi)z  \geq  -\e^*(\G_u(t,\theta)-\G_u(t,0)-\lambda_{0,\G}) \geq \frac{\lambda_{0,\G}\e^*}{2}>0.
\end{equation*}
Furthermore, according to \eqref{demcomp1}, we have $z\geq 0$ on $\Omega$ and even on $\R \times [R^+,+\infty)$. As $z(t^*,\xi^*)=0$ and $(t^*,\xi^*)\in \Omega$, it follows from \cite[Th2 p168]{PW} that
\begin{equation}
\label{protter}
z \equiv 0~\text{on}~\Omega_0,
\end{equation}
where $\Omega_0$ is the set of point $(t,\xi)\in \Omega$ such that there exists a continuous path $\gamma:[0,1] \to \Omega$ such that $\gamma(0)=(t,\xi)$ and $\gamma(1)=(t^*,\xi^*)$, and the time component of $\gamma(s)$ is nondecreasing with respect to $s\in [0,1]$.
Now, we define $$\underline{\xi}=\inf\big\{\xi\geq R_+~|~\forall \tilde{\xi} \in (\xi,\xi^*],~~(t^*,\tilde{\xi}) \in \Omega_0\big\}.$$
The equality \eqref{protter} and the continuity of $z$ imply that 
$$z(t^*,\underline{\xi})=0.$$ 
So, according to \eqref{comp21}, we have $\underline{\xi}>R_+$ (because $\V^*(t^*,R_+)\geq \v^*(t^*,R_+)$). Consequently, the definition of $\underline{\xi}$ and the continuity of $\V^*$ and $\v^*$ imply that 
$$\V^*(t^*,\underline{\xi})=\v^*(t^*,\underline{\xi}).$$
Finally, the previous two displayed equalities imply that $\e^*=0$, which contradicts the fact that $\e^*$ is positive. Consequently $\e^*=0$ and the lemma is proved.
\end{proof}
\noindent
With a similar proof, we can obtain the following lemma.
\begin{lemma}
\label{TH2comparaison}
Let us fix $c\in \R$, $R_- \in \R$ and $\alpha \in (0,1)$. We consider two functions $\g$ and $\G$ of class $\mathcal{C}^1(\R \times [0,1],\R)$, T-periodic and such that  
\begin{equation*}
\label{compf2}
\g(t,u)\leq \G(t,u),~~~~\forall t\in [0,T],~~\forall u \in [0,1].
\end{equation*}
\vspace{-0.5em}\\
We assume $\g$ satisfies  the hypothesis  \eqref{01} and the second inequality of \eqref{bismoyT}.
\vspace{0.5em}\\
Suppose there exist two functions $\V:\R \times (-\infty,R_-] \to [0,1],~(t,\xi) \mapsto \V(t,\xi)$ and $\v:\R \times (-\infty,R_-] \to [0,1],~(t,\xi) \mapsto \v(t,\xi)$ of class $\mathcal{C}^{1,\frac{\alpha}{2}}(\R)$ in $t$ uniformly for $\xi \in (-\infty,R_-]$ and of class $\mathcal{C}^{2,\alpha}((-\infty,R_-])$ in $\xi$ uniformly for $t \in \R,$ and such that 
\begin{equation*}
\label{usursol2}
\partial_t \V-c\partial_\xi \V-\partial_{\xi \xi} \V \geq \G(t,\V)~~\text{on}~\R \times (-\infty,R_-],
\end{equation*}
\begin{equation*}
\label{usoussol2}
\partial_t \underline{v}-c\partial_\xi \underline{v}-\partial_{\xi \xi} \underline{v} \leq \underline{g}(t,\underline{v})~~\text{on}~\R \times (-\infty,R_-],
\end{equation*}
\begin{equation*}
\v(t,R_-) \leq \V(t,R_-),~~\forall t \in \R,
\end{equation*}
\begin{equation*}
\label{limunifu2}
\V(t,-\infty)=1~~\text{uniformly on}~\R.
\end{equation*}
There exists $\delta_- \in (0,1)$ depending only on $\g$ such that, if we have 
\begin{equation*}
\V(t,\xi)\geq 1-\delta_-, ~~\forall t \in \R,~~\forall \xi \in (-\infty,R_-], 
\end{equation*}
then
\begin{equation*}
\label{thcomparaison2}
\v(t,x)\leq \V(t,x)~~\forall t\in \R,~~\forall \xi \in (-\infty,R_-].
\end{equation*}
\end{lemma}

\subsection{Monotonicity of the front}
Let us consider $(U,c)$ a solution of \eqref{travel}, with $U:\R\times \R \to [0,1]$. 
We want to prove that for all $\tau \geq 0$, we have 
$$U(t,\xi) \leq U(t,\xi-\tau),~~~\forall t \in \R,~~\forall \xi \in \R.$$
We start with the following lemma.
\begin{lemma}
\label{AZ}
There exists $\tau_0 \geq 0$ such that for any $\tau \geq \tau_0$, we have 
$$U(t,\xi) \leq U(t,\xi-\tau),~~~\forall t \in \R,~~\forall \xi \in \R.$$
\end{lemma}
\begin{proof}
As $U(\cdot,+\infty)=0$ uniformly on $\R$, there exists a real $R_+$ such that 
\begin{equation}
\label{QS}
U(t,\xi) \leq \delta_+,  ~~~~\forall t \in \R,~~\forall \xi \in [R_+,+\infty),
\end{equation}
with $\delta_+ \in (0,1)$ defined in Lemma \ref{TH1comparaison} with $\G=\g=f$.\\
As $U(\cdot,-\infty)=1$ uniformly on $\R$, there exists $\tau_0\geq0$ such that 
$$U(t,\xi-\tau) \geq 1-\delta_-,~~~~\forall \tau \geq \tau_0,~~\forall t \in \R,~~\forall \xi \in (-\infty,R_+],$$
where $\delta_- \in (0,1)$ is defined in Lemma \ref{TH2comparaison} with $\G=\g=f$.\\
Without loss of generality, we can assume that $\max{\{ \delta_-,\delta_+\}}<1/2$, whence $U(t,R_+) \leq U(t,R_+-\tau)$ for all $t \in \R$ and $\tau \geq \tau_0$. We can apply Lemma \ref{TH1comparaison} on $\R \times [R_+,+\infty)$ and Lemma \ref{TH2comparaison} on $\R \times (-\infty,R_+]$ to the functions $\v=U$ and $\V=U(\cdot,\cdot-\tau)$ for any $\tau \geq \tau_0$, and the lemma is proved.
 \end{proof}
\begin{proposition}
\label{MONOT}
We have 
$$\partial_\xi U(t,\xi)<0,~~~~\forall t\in \R,~~\forall \xi \in \R.$$
\end{proposition}
\begin{proof}
We define $R_+$ as in the previous proof. As $U(\cdot,-\infty)=1$ uniformly on $\R$, there exists $R_- \leq R_+$ such that 
$$U(t,\xi) \geq 1-\delta_-,~~~~\forall t \in \R,~~\forall \xi \in (-\infty,R_-].$$
Consequently,
\begin{equation}
\label{QSD}
U(t,\xi-\tau) \geq 1-\delta_-,~~~~\forall \tau \geq 0,~~\forall t \in \R,~~\forall \xi \in (-\infty,R_-].
\end{equation}
We define 
$$\tau^*=\inf \big\{ \tilde{\tau} \geq 0~|~U(t,\xi) \leq U(t,\xi-\tau),~\forall \tau \geq \tilde{\tau},~~\forall (t,\xi) \in \R^2 \big\}.$$
The constant $\tau^*$ is a well defined real number according to Lemma \ref{AZ}. We will prove that $\tau^*=0$ by contradiction. Let us suppose that $\tau^*>0$. The definition of $\tau^*$ and the continuity of $U$ imply that 
\begin{equation}
\label{inegalite}
U(t,\xi) \leq U(t,\xi-\tau^*),~~~~\forall t \in \R,~~\forall \xi \in \R.
\end{equation}
We define $$\eta:=\min\limits_{(t,\xi)\in [0,T] \times [R_-,R_+]}\{U(t,\xi-\tau^*)-U(t,\xi)\}.$$\\
Two cases can occur, either $\eta>0$, or $\eta=0$.
\vspace{1em}\\
$\textit{1}^{st}~case:~\eta>0$. By uniform continuity of $U$ on $[0,T] \times [R_--\tau^*,R_+]$, there exists  $\overline{\tau}\in(0,\tau^*)$ such that 
\begin{equation}
\label{QSDF}
U(t,\xi) < U(t,\xi-\tau),~~~\forall \tau \in [\overline{\tau},\tau^*],~~\forall t\in[0,T],~~\forall \xi \in [R_-,R_+].
\end{equation}
Let $\tau \in [\overline{\tau},\tau^*]$. According to \eqref{QS} and \eqref{QSDF}, and since $U$ is $T$-periodic in $t$,  we can apply Lemma \ref{TH1comparaison} to  $U$ and $U(\cdot,\cdot-\tau)$ on $\R \times [R_+,+\infty)$. We obtain 
$$U(t,\xi) \leq U(t,\xi-\tau)~~~\forall t \in[0,T],~~\forall \xi \in [R_+,+\infty).$$
And according to \eqref{QSD} and \eqref{QSDF}, we can apply Lemma \ref{TH2comparaison} to $U$ and $U(\cdot,\cdot-\tau)$ on $\R \times (-\infty,R_-]$. We obtain 
$$U(t,\xi) \leq U(t,\xi-\tau),~~~\forall t \in[0,T],~~ \forall \xi \in (-\infty,R_-].$$
To summarize $U(t,\xi) \leq U(t,\xi-\tau)$ for all $t \in[0,T]$, and $\xi \in \R.$
Consequently, since $U$ is $T$-periodic in $t$, we have 
$$U(t,\xi) \leq U(t,\xi-\tau),~~~\forall \tau \in [\overline{\tau},\tau^*]~~\forall t \in \R,~~\forall \xi \in \R.$$
That is
$$U(t,\xi) \leq U(t,\xi-\tau),~~~\forall \tau \geq\overline{\tau},~~\forall t \in \R,~\forall \xi \in \R.$$ 
This contradicts the definition of $\tau^*$.
\vspace{1em}\\
$\textit{2}^{nd}~case:~\eta=0$. Let us begin by noting that there exists a couple $(t^*,\xi^*)\in [0,T] \times [R_-,R_+]$ such that $\eta=U(t^*,\xi^*-\tau^*)-U(t^*,\xi^*).$ Consequently, we have $U(t^*,\xi^*)=U(t^*,\xi^*-\tau^*).$ By applying the strong parabolic maximum principle on $\R^2$, we infer that
\begin{equation*}
U(t,\xi)=U(t,\xi-\tau^*)~~\forall t \in (-\infty,t^*],~~\forall \xi \in \R.
\end{equation*}
This implies that $U(t^*,\cdot)$ is a periodic function, which is impossible since $U(\cdot,-\infty)=1$ and $U(\cdot,+\infty)=0$.
\vspace{1em}\\
So we have $\tau^*=0$, that is 
$$U(t,\xi) \leq U(t,\xi-\tau),~~~~\forall \tau \geq 0,~~\forall t \in \R,~~\forall \xi \in \R.$$
In other terms $\partial_\xi U(t,\xi) \leq 0$ for all $t \in \R$ and $\xi \in \R.$ We apply the strong maximum principle to the equation satisfied by $\partial_\xi U$ and obtain 
$\partial_\xi U(t,\xi) < 0$ for all $(t,\xi) \in \R^2$ (otherwise $\partial_\xi U$ would be identically equal to zero, which is impossible since $U(\cdot,-\infty)=1$ and $U(\cdot,+\infty)=0$). 
\end{proof}
\subsection{Uniqueness of $(U,c)$}
\noindent
We consider $(U_1,c_1)$ and $(U_2,c_2)$ two solutions of problem \eqref{travel}, with $U_1,U_2: \R \times \R \to [0,1]$. Without loss of generality, we can assume $c_1\geq c_2$.
\begin{lemma}
\label{AZE}
There exists $\tau_0 \geq 0$ such that for any $\tau \geq \tau_0$, we have 
$$U_1(t,\xi) \leq U_2(t,\xi-\tau),~~~\forall t \in \R,~~\forall \xi \in \R.$$
\end{lemma}
\begin{proof}
The function $U_2$ is supersolution of problem \eqref{travel} with speed $c_1$ since $\partial_\xi U_2<0$ (Prop. \ref{MONOT}) and $c_1\geq c_2$. Indeed 
$$\partial_t U_2 - c_1 \partial_\xi U_2 - \partial_{\xi\xi}U_2 = f^T(t,U_2)+(c_2-c_1)\partial_\xi U_2\geq f^T(t,U_2)~~~~\text{on}~\R^2.$$
 As $U_1(\cdot,+\infty)=0$ uniformly on $\R$, there exists a real number $R_+$ such that 
\begin{equation}
\label{WX}
U_1(t,\xi) \leq \delta_+,~~~~ \forall t \in \R,~\forall \xi \in [R_+,+\infty),
\end{equation}
where $\delta_+ \in (0,1)$ is given by Lemma \ref{TH1comparaison} with $\G=\g=f$ and only depends on $f$..\\
As $U_2(\cdot,-\infty)=1$ uniformly on $\R$, there exists $\tau_0\geq0$ such that for any $\tau \geq \tau_0$, we have 
$$U_2(t,\xi-\tau) \geq 1-\delta_-, ~~~~\forall t \in \R,~~\forall \xi \in (-\infty,R_+],$$
where $\delta_- \in (0,1)$ is given by Lemma \ref{TH2comparaison} with $\G=\g=f$ and only depends on $f$.\\
We can apply Lemma \ref{TH1comparaison} on $\R \times [R_+,+\infty)$ and Lemma \ref{TH2comparaison} on $\R \times (-\infty,R_+]$ to the functions $\v=U_1$ and $\V=U_2(\cdot,\cdot-\tau)$ for any $\tau \geq \tau_0$ to obtain the lemma.
\end{proof}
\begin{proposition}
\label{PPRROO}
There exists $\tau^* \in \R$ such that 
$$U_1(t,\xi) = U_2(t,\xi-\tau^*),~~~~\forall t\in \R,~~\forall \xi\in \R.$$
And thus $c_1=c_2.$
\end{proposition}
\begin{proof}
We define 
$$\tau^*=\inf \big\{ \tau \in \R~|~U_1(t,\xi) \leq U_2(t,\xi-\tau)~\forall (t,\xi) \in  \R^2 \big\}.$$
The set inside infimum is not empty according to Lemma \ref{AZE}. Furthermore, it is bounded from below because as $U_2(0,+\infty) = 0$ and $U_1(0, 0) \in(0, 1)$, we can find a small enough $\tau_1$ so that for any $\tau \leq \tau_1$, we have $U_2(0, -\tau) < U_1(0, 0)$.
Consequently, $\tau^*$ is a well defined real number. The definition of $\tau^*$ and the continuity of $U_1$ and $U_2$ imply that 
\begin{equation}
\label{Aaa}
U_1(t,\xi) \leq U_2(t,\xi-\tau^*),~~~~\forall t \in \R,~~\forall \xi \in \R.
\end{equation}
We define $R_+$ as in the proof of Lemma \ref{AZE}. As $U_2(\cdot,-\infty)=1$ uniformly on $\R$, there exists $R_- \leq R_+$ such that 
\begin{equation}
\label{WXC}
U_2(t,\xi-\tau) \geq 1-\delta_-,~~~~\forall \tau \geq \tau^*-1,~~\forall t \in \R,~~\forall \xi \in (-\infty,R_-].
\end{equation}
We define $$\eta:=\min\limits_{(t,\xi)\in [0,T] \times [R_-,R_+]}\big\{U_2(t,\xi-\tau^*)-U_1(t,\xi)\big\}.$$
Notice that $\eta \geq 0$ by \eqref{Aaa}. If $\eta>0$, by the uniform continuity of the function $U_2$ on $[0,T] \times [R_--\tau^*,R_+-\tau^*+1]$, there exists  $\overline{\tau}\in(\tau^*-1,\tau^*)$ such that 
\begin{equation}
\label{WXCV}
U_1(t,\xi) \leq U_2(t,\xi-\overline{\tau}),~~~\forall t \in[0,T], ~\forall \xi \in [R_-,R_+].
\end{equation}
According to \eqref{WX}, \eqref{WXCV} and the $T$-periodicity of $U_1$ and $U_2$ in $t$, we can apply Lemma \ref{TH1comparaison} to the functions $\v=U_1$ and $\V=U_2(\cdot,\cdot-\overline{\tau})$ on $\R \times [R_+,+\infty)$. We obtain 
$$U_1(t,\xi) \leq U_2(t,\xi-\overline{\tau}),~~~\forall t \in \R,~~ \forall \xi \in [R_+,+\infty).$$
In the same way,  according to \eqref{WXC} and \eqref{WXCV}, we can apply Lemma \ref{TH2comparaison} to the functions $\v=U_1$ and $\V=U_2(\cdot,\cdot-\tau^*)$ on $\R \times (-\infty,R_-]$. We obtain 
$$U_1(t,\xi) \leq U_2(t,\xi-\overline{\tau}),~~~\forall t \in \R,~~ \forall \xi \in (-\infty,R_-].$$
To summarize $U_1(t,\xi) \leq U_2(t,\xi-\overline{\tau})$ for all $t \in[0,T]$ and $\xi \in \R.$ This contradicts the definition of $\tau^*$. Consequently, we have $\eta=0$. So, there exists a couple $(t^*,\xi^*)\in [0,T] \times [R_-,R_+]$ such that 
$$U_1(t^*,\xi^*)=U_2(t^*,\xi^*-\tau^*).$$
By applying the strong maximum principle on $ \R^2$, we get that \\
\begin{equation*}
U_1(t,\xi)=U_2(t,\xi-\tau^*),~~\forall t \in (-\infty, t^*],~~\forall \xi \in \R.
\end{equation*}
By periodicity in $t$, the previous inequality is true on $ \R^2$.
\vspace{1em}\\
Now, if we substract the equation satisfied by $U_1$ to the equation satisfied by $U_2(\cdot,\cdot-\tau^*)$, we obtain that 
$$(c_2-c_1)\partial_\xi U_2(t,\xi-\tau^*)=0,~~~~\forall t \in \R,~~\forall \xi \in  \R.$$
Since $\partial_\xi U_2<0$, we get $c_1=c_2$ and the proof of Proposition \ref{PPRROO} is complete.
\end{proof}
\section{Asymptotic stability of pulsating waves}
This section is devoted to the proof of Theorem \ref{stability}. To simplify the notations, we note $\lambda_0$ (resp. $\lambda_1$)  the principal eigenvalue associated with the function $f^T$ and the equilibrium $0$ (resp. $1$). The condition \eqref{bismoyT} implies that $\lambda_0>0$ and $\lambda_1>0$. Furthermore, according to Section \ref{characterization}, there exists a unique positive $T$-periodic function $\Phi_0(t)$ such that 
\begin{equation*}
\begin{cases}
\Phi_0'(t)=\big(\lambda_0+f^T_u(t,0)\big)\Phi_0(t),~~\forall t \in \R, \\
\Phi_0(0)=1.
\end{cases}
\end{equation*}
Also, there exists a unique positive $T$-periodic function $\Phi_1(t)$ such that 
\begin{equation*}
\begin{cases}
\Phi_1'(t)=\big(\lambda_1+f^T_u(t,1)\big)\Phi_1(t),~~\forall t \in \R, \\
\Phi_1(0)=1.
\end{cases}
\end{equation*}
Assume there exists a pulsating front $U$ with speed $c$ solving \eqref{travel}. We consider a solution $u$ of the Cauchy problem 
$$\begin{cases}
u_t-u_{xx}=f^T(t,u)~~\text{on}~(0,+\infty)\times\R, \\
u(0,x)=h(x),~~\forall x \in \R,
\end{cases}$$
where the initial condition $h:\R \to [0,1]$  is uniformly continuous. We denote $v(t,\xi)=u(t,x)=u(t,\xi+ct)$. The function $v$ satisfies the Cauchy problem
\begin{equation}
\label{systv}
\begin{cases} \partial_tv-c\partial_\xi v-\partial_{\xi \xi}v=f^T(t,v)~~\text{on}~(0,+\infty) \times \R, \\
 v(0,\xi)=h(\xi),~~\forall \xi \in \R. 
 \end{cases}
 \end{equation}
Our basic lemma is the following.
\begin{lemma}
\label{lemstab18}
There exists a constant $\gamma\in (0,1)$ depending only on $f^T$ and $U$, such that if 
$$\liminf\limits_{\xi \to -\infty}h(\xi)>1-\gamma~~\text{and}~~\limsup\limits_{\xi \to +\infty}h(\xi)<\gamma,$$
then, there exist some real numbers $q_0>0$, $\underline{\xi}^0$ and $\overline{\xi}^0$, such that 
\begin{equation}
\label{soussur8}
 U(t,\xi+\underline{\xi}^0)-q_0(\|\Phi_0\|_\infty+\|\Phi_1\|_\infty)e^{-\mu t}  \leq v(t,\xi) \leq U(t,\xi+\overline{\xi}^0)+q_0(\|\Phi_0\|_\infty+\|\Phi_1\|_\infty)e^{-\mu t},
 \end{equation}
for any real number $\xi$ and for any $t$ positive.
\end{lemma}
\begin{proof}
We are only going to prove the left inequality, the other is similar. We begin by defining parameters which are independant of $h$ and from which we are going to build a subsolution of \eqref{systv} on $\R^+ \times \R$. We define $$\mu=\min\Big\{\frac{\lambda_0}{2},\frac{\lambda_1}{2}\Big\}>0.$$
The regularity of the function $f^T$ implies that there exists $u_0>0$ such that 
\begin{equation}
\label{stab18}
\forall u \in [0,u_0],~\forall t \in \R,~~
\begin{cases}
|f^T_u(t,u)-f^T_u(t,0)|\leq \frac{\lambda_0-\mu}{4},\\
|f^T_u(t,1)-f^T_u(t,1-u)|\leq \frac{\lambda_1-\mu}{4}.
\end{cases}
\end{equation}
Since $U(\cdot,+\infty)=0$ and $U(\cdot,-\infty)=1$ uniformly on $\R$, there exist $\xi^+>0$ and $\xi^-<0$ such that 
\begin{equation}
\label{stab28}
\forall t\in \R,~~\forall \tilde{\xi}\geq \xi^+,~~~~U(t,\tilde{\xi})\leq u_0,
\end{equation}
\begin{equation}
\label{stab38}
\forall t\in \R,~~\forall \tilde{\xi}\leq \xi^-,~~~~1-U(t,\tilde{\xi})\leq u_0.
\end{equation}
We consider a nondecreasing function $\chi$ in $\mathcal{C}^2(\R)$ such that 
$$\forall \tilde{\xi}\geq \xi^+,~~\chi(\tilde{\xi})=1~~\text{and}~~\forall \tilde{\xi}\leq \xi^- ,~\chi(\tilde{\xi})=0.$$
According to the fact that the function $U_\xi$ is continuous, negative and $T$-periodic, one has $\sup_{\R \times [\xi^-,\xi^+]}U_\xi>0$, and there exist $\tilde{q}_0>0$ and $C_1>0$ (depending only on $f^T$ and $U$) such that 
\begin{equation}
\label{forall}
\forall q \in [0,\tilde{q}_0],~\forall t\in \R^+,~\forall \tilde{\xi} \in [\xi^-,\xi^+],~~U_\xi(t,\tilde{\xi})+qe^{-\mu t}\chi'(\tilde{\xi})\big(\Phi_1(t)-\Phi_0(t)\big)<-C_1.
\end{equation}
There exists $C_2>0$ such that
\begin{equation}
\label{forall1}
\forall t \in \R,~\forall \tilde{\xi} \in [\xi^-,\xi^+],~~\chi(\tilde{\xi})\Phi_0(t)+\big(1-\chi(\tilde{\xi})\big)\Phi_1(t)< \frac{C_2}{\|f_u^T\|_{\infty}}. \end{equation}
There exists $C_3>0$ such that
\begin{multline}
\label{forall2}
\forall t \in \R,~\forall \tilde{\xi} \in [\xi^-,\xi^+],~~\chi(\tilde{\xi})\Phi_0(t)\big( \mu-\lambda_0-f_u(t,0)  \big)+\big(1-\chi(\tilde{\xi})\big)\Phi_1(t)\big( \mu-\lambda_1-f_u(t,1)\big)\\
+\big( c\chi'(\tilde{\xi})+\chi''(\tilde{\xi}) \big)\big(\Phi_0(t)-\Phi_1(t)\big)<C_3.
\end{multline}
We define
\begin{equation}
\label{defomega}
\om=\frac{C_2+C_3}{\mu C_1}>0.
\end{equation}
From the continuity and the $T$-periodicity of $f^T_u$, there exists $\gamma \in (0,\min\{\tilde{q}_0,1\})$ (depending only on $f^T$ and $U$) such that 
\begin{multline}
\label{forall3}
\forall t\in\R,~\forall (u,u')\in [0,1]^2,~~|u-u'| \leq \gamma \max\{\|\Phi_0\|_{\infty},\|\Phi_1\|_\infty\}  \Rightarrow \\ |f^T(t,u)-f^T(t,u')-f^T_u(t,u')(u-u')|
 \leq \min \Big\{ \frac{\lambda_0-\mu}{4},\frac{\lambda_1-\mu}{4} \Big\} |u-u'|.
\end{multline}
We now consider a uniformly continuous function $h:\R \to [0,1]$ such that 
$$\liminf\limits_{\xi \to -\infty}h(\xi)>1-\gamma~~\text{and}~~\limsup\limits_{\xi \to +\infty}h(\xi)<\gamma,$$
and we denote $v$ the solution of \eqref{systv}. We consider a real number $q_0>0$ such that $1-\liminf\limits_{\xi \to -\infty}h(\xi)<q_0<\gamma.$ We thus have $\liminf\limits_{\xi \to -\infty}h(\xi)>1-q_0.$ Consequently, there exists $\xi_m \in \R$ such that for all $\xi \leq \xi_m$ we have $h(\xi)\geq 1-q_0$. Yet, since $h\geq 0$ and $U(0,+\infty)=0$, there exists $\underline{\xi} \in \R$ such that for all $\xi \geq \xi_m$, we have $h(\xi) \geq U(0,\xi+\underline{\xi})-q_0.$ According to the fact that $U\leq1$ and that for all $\xi \in \R$ we have $\chi(\xi+\underline{\xi})\Phi_0( 0 )+(1-\chi(\xi+\underline{\xi}))\Phi_1(0)=1$, we have that
\begin{equation}
\label{compenzero}
\forall \xi \in \R,~~~h(\xi) \geq U(0,\xi+\underline{\xi})-q_0\Big[\chi(\xi+\underline{\xi})\Phi_0(0)+\{1-\chi(\xi+\underline{\xi})\}\Phi_1(0)\Big].
\end{equation}
Finally, we define 
$$\underline{\Lambda}(t)=\om q_0 (1-e^{-\mu t}) +\underline{\xi}.$$
We are going to show that the function defined by
$$\u(t,\xi) =\max\Big\{0, U\big(t,\xi+\underline{\Lambda}(t)\big) -q_0e^{-\mu t}\Big[ \chi\big(\xi+\underline{\Lambda}(t)\big)\Phi_0(t)+\big\{1-\chi\big(\xi+\underline{\Lambda}(t)\big)\big\}\Phi_1(t)\Big]\Big\}.$$
is a subsolution of \eqref{systv} on $\R^+\times \R$. We begin by noticing that according to \eqref{compenzero}, for any $\xi \in \R$, we have $\underline{u}(0,\xi) \leq v(0,\xi)$. We divide now the space into three zones:
$$\Omega^-=\big\{(t,\xi)\in(0,+\infty)\times\R~|~\xi+\underline{\Lambda}(t)<\xi^-, ~\text{and}~\u(t,\xi)>0\big\},$$
$$\Omega^+=\big\{(t,\xi)\in(0,+\infty)\times\R~|~\xi+\underline{\Lambda}(t)>\xi^+, ~\text{and}~\u(t,\xi)>0\big\},$$
$$\Omega^0=\big\{(t,\xi)\in(0,+\infty)\times\R~|~\xi+\underline{\Lambda}(t)\in[\xi^-,\xi^+], ~\text{and}~\u(t,\xi)>0\big\}.$$
Since $f^T(t,0)=0$, we only have to show that $\underline{u}_t- c\underline{u}_{\xi}-\underline{u}_{\xi\xi}-f^T(t,\underline{u})\leq0$ on $\Omega^+\cup \Omega^- \cup \Omega^0$.\\
$\textit{1}^{st}~step$: we show that $\underline{u}_t- c\underline{u}_{\xi}-\underline{u}_{\xi\xi}-f^T(t,\underline{u})\leq0$ on $\Omega^+$. We have
$$\forall (t,\xi)\in\Omega^+,~~\u(t,\xi) = U\big(t,\xi+\underline{\Lambda}(t)\big) -q_0e^{-\mu t}\Phi_0(t) .$$
Consequently, for all $(t,\xi)\in\Omega^+$, then we have
\begin{multline*}
$$\underline{u}_t(t,\xi)- c\underline{u}_{\xi}(t,\xi)-\underline{u}_{\xi\xi}(t,\xi)-f^T\big(t,\underline{u}(t,\xi)\big)=f^T\big(t,U\big(\xi+\underline{\Lambda}(t)\big)\big)-f^T\big(t,\u(t,\xi)\big)\\
+q_0e^{-\mu t}\Phi_0(t)\big[\mu-\lambda_0-f^T_u(t,0)\big]+\underline{\Lambda}'(t)U_\xi\big(t,\xi+\underline{\Lambda}(t)\big).
\end{multline*}
We have
\begin{equation*}
\begin{array}{l}
f^T\big(t,U\big(t,\xi+\underline{\Lambda}(t)\big)\big)-f^T\big(t,\u(t,\xi)\big)\\
 =\Big\{f^T\big(t,U\big(t,\xi+\underline{\Lambda}(t)\big)\big)-f^T\big(t,\underline{u}(t,\xi)\big)+f^T_u\big(t,U\big(t,\xi+\underline{\Lambda}(t)\big)\big)\big[\underline{u}(t,\xi)-U\big(t,\xi+\underline{\Lambda}(t)\big)\big]\Big\}\\
  +\Big\{\big[f^T_u\big(t,U\big(t,\xi+\underline{\Lambda}(t)\big)\big)-f^T_u(t,0)\big]\big[U\big(t,\xi+\underline{\Lambda}(t)\big)-\underline{u}\big(t,\xi\big)\big]\Big\}\\
  +f^T_u(t,0)\Big\{U\big(t,\xi+\underline{\Lambda}(t)\big)-\underline{u}(t,\xi)\Big\}.
\end{array}
\end{equation*}
According to the fact that $0<q_0 \leq \gamma$, \eqref{forall3} yields 
\begin{multline*}
f^T\big(t,U\big(t,\xi+\underline{\Lambda}(t)\big)\big)-f^T\big(t,\underline{u}(t,\xi)\big)+f^T_u\big(t,U\big(t,\xi+\underline{\Lambda}(t)\big)\big)\big[\underline{u}(t,\xi)-U\big(t,\xi+\underline{\Lambda}(t)\big)\big]\\
\leq \frac{\lambda_0-\mu}{4}q_0\Phi_0(t)e^{-\mu t}.
\end{multline*}
Furthermore, according to $\eqref{stab18}~\text{and}~\eqref{stab28}$, we have 
$$\big[f^T_u\big(t,U\big(t,\xi+\underline{\Lambda}(t)\big)\big)-f^T_u(t,0)\big]\big[U\big(t,\xi+\underline{\Lambda}(t)\big)-\underline{u}\big(t,\xi\big)\big] \leq \frac{\lambda_0-\mu}{4}q_0\Phi_0(t)e^{-\mu t}.$$
So, since $\underline{\Lambda}' \geq0$ and $U_\xi<0$, we have 
$$\underline{u}_t(t,\xi)- c\underline{u}_{\xi}(t,\xi)-\underline{u}_{\xi\xi}(t,\xi)-f^T\big(t,\underline{u}(t,\xi)\big)\leq\big[\mu-\lambda_0+\frac{\lambda_0-\mu}{4}+\frac{\lambda_0-\mu}{4}\big]q_0\Phi_0(t)e^{-\mu t}\leq0.$$
$\textit{2}^{nd}~step$: we show that $\underline{u}_t- c\underline{u}_{\xi}-\underline{u}_{\xi\xi}-f^T(t,\underline{u})\leq0$ on $\Omega^-$. We have
$$\forall (t,\xi)\in\Omega^-,~~\u(t,\xi) =U\big(t,\xi+\underline{\Lambda}(t)\big) -q_0e^{-\mu t}\Phi_1(t).$$
Hence, for all $(t,\xi)\in\Omega^-$,
\begin{multline*}
\underline{u}_t(t,\xi)- c\underline{u}_{\xi}(t,\xi)-\underline{u}_{\xi\xi}(t,\xi)-f^T\big(t,\underline{u}(t,\xi)\big)=f^T\big(t,U\big(\xi+\underline{\Lambda}(t)\big)\big)-f^T\big(t,\u\big(t,\xi\big)\big)\\
+q_0e^{-\mu t}\Phi_1(t)\big[\mu-\lambda_1-f^T_u(t,1)\big]+\underline{\Lambda}'(t)U_\xi\big(t,\xi+\underline{\Lambda}(t)\big).
\end{multline*}
In the same way as previously, since $\underline{\Lambda}' \geq 0$ and $U_\xi<0$, we have that 
$$\underline{u}_t(t,\xi)- c\underline{u}_{\xi}(t,\xi)-\underline{u}_{\xi\xi}(t,\xi)-f^T\big(t,\underline{u}(t,\xi)\big)\leq\big[\mu-\lambda_1+\frac{\lambda_1-\mu}{4}+\frac{\lambda_1-\mu}{4}\big]q_0\Phi_1(t)e^{-\mu t}\leq0.$$
$\textit{3}^{rd}~step$: we show that $\underline{u}_t- c\underline{u}_{\xi}-\underline{u}_{\xi\xi}-f^T(t,\underline{u})\leq0$ on $\Omega^0$. For all $(t,\xi)\in\Omega^0$ there holds
\begin{equation}
\label{tableau}
\begin{array}{l}
\underline{u}_t(t,\xi)- c\underline{u}_{\xi}(t,\xi)-\underline{u}_{\xi\xi}(t,\xi)-f^T\big(t,\underline{u}(t,\xi)\big)\\
\qquad \qquad \qquad=f^T\big(t,U\big(t,\xi+\underline{\Lambda}(t)\big)\big)-f^T\big(t,\underline{u}(t,\xi)\big)\\
\qquad \qquad \qquad+q_0e^{-\mu t}\Phi_0(t)\chi\big(\xi+\underline{\Lambda}(t)\big)\big\{\mu-\lambda_0-f^T_u(t,0)\big\} \\
\qquad \qquad \qquad+q_0e^{-\mu t}\Phi_1(t)\big\{1-\chi\big(\xi+\underline{\Lambda}(t)\big)\big\}\big\{\mu-\lambda_1-f^T_u(t,1)\big\}\\
\qquad \qquad \qquad+q_0e^{-\mu t}\big\{c \chi'\big(\xi+\underline{\Lambda}(t)\big)+ \chi''\big(\xi+\underline{\Lambda}(t)\big)\big\}\big\{\Phi_0(t)-\Phi_1(t)\big\}\\
\qquad \qquad \qquad+\underline{\Lambda}'(t)\big\{ U_\xi\big(t,\xi+\underline{\Lambda}(t)\big)+q_0e^{-\mu t}\chi'\big(\xi+\underline{\Lambda}(t)\big)\big(\Phi_1(t)-\Phi_0(t)\big)\big\}.
\end{array}
\end{equation}
We have 
$$f^T\big(t,U\big(t,\xi+\underline{\Lambda}(t)\big)\big)-f^T\big(t,\underline{u}\big(t,\xi\big)\big)\leq \|f^T_u\|_\infty \big|U\big(t,\xi+\underline{\Lambda}(t)\big)-\underline{u}\big(t,\xi\big)\big|.$$
According to \eqref{forall1} applied with $\tilde{\xi}=\xi+\underline{\Lambda}(t)$, it occurs that
\begin{equation}
\label{avecC2}
f^T\big(t,U\big(t,\xi+\underline{\Lambda}(t)\big)\big)-f^T\big(t,\underline{u}\big(t,\xi\big)\big)\leq C_2q_0e^{-\mu t}.
\end{equation}
According to \eqref{forall2} applied with $\tilde{\xi}=\xi+\underline{\Lambda}(t)$, we have 
\begin{equation}
\begin{array}{l}
\label{avecC3}
q_0e^{-\mu t}\Phi_0(t)\chi\big(\xi+\underline{\Lambda}(t)\big)\Big\{\mu-\lambda_0-f^T_u(t,0)\Big\} \\
\qquad \qquad \qquad+q_0e^{-\mu t}\Phi_1(t)\big\{1-\chi\big(\xi+\underline{\Lambda}(t)\big)\big\}\big\{\mu-\lambda_1-f^T_u(t,1)\big\}\\
\qquad \qquad \qquad+q_0e^{-\mu t}\big\{c \chi'\big(\xi+\underline{\Lambda}(t))+ \chi''\big(\xi+\underline{\Lambda}(t)\big)\big\}\big\{\Phi_0(t)-\Phi_1(t)\big\} \\
\qquad \qquad \qquad\leq C_3 q_0 e^{-\mu t}.
\end{array}
\end{equation}
Finally, remembering that $\underline{\Lambda}'>0$ and according to \eqref{forall} applied with $\tilde{\xi}=\xi+\underline{\Lambda}(t)$, it follows that 
\begin{equation}
\label{avecC1}
\big\{ U_\xi\big(t,\xi+\underline{\Lambda}(t)\big)+q_0e^{-\mu t}\chi'\big(\xi+\underline{\Lambda}(t)\big)\big(\Phi_1(t)-\Phi_0(t)\big)\big\}\underline{\Lambda}'(t)<-C_1\underline{\Lambda}'(t).
\end{equation}
Consequently, according to \eqref{defomega}, \eqref{tableau}, \eqref{avecC2}, \eqref{avecC3} and \eqref{avecC1} we obtain for all $(t,\xi)\in \Omega_0$,
$$\underline{u}_t(t,\xi)- c\underline{u}_{\xi}(t,\xi)-\underline{u}_{\xi\xi}(t,\xi)-f^T\big(t,\underline{u}(t,\xi)\big) \leq -C_1\underline{\Lambda}'(t)+(C_2+C_3)q_0e^{-\mu t}\leq 0.$$
We conclude from the maximum principle that \\
$$\forall t \in \R^+,~~\forall \xi \in \R,~~~~v(t,\xi)\geq\underline{u}(t,\xi).$$
So, as $U_\xi<0$ on $\R^2$ and $\underline{\Lambda}'\geq0$ on $\R^+$, if we define the real number $\underline{\xi}^0=\underline{\xi}+\om q_0$, it occurs that for all $t\geq0$ and $\xi\in \R$
\begin{align*}
v\big(t,\xi\big) & \geq   U\big(t,\xi+\underline{\Lambda}(t)\big)-q_0e^{-\mu t}\big[\chi\big(\xi+\underline{\Lambda}(t)\big)\Phi_0(t)+\big\{1-\chi\big(\xi+\underline{\Lambda}(t)\big)\big\}\Phi_1(t)\big]\\
&\geq  U\big(t,\xi+\underline{\xi}^0\big)-q_0\big(\|\Phi_0\|_\infty+\|\Phi_1\|_\infty\big)e^{-\mu t}.
\end{align*}
It is exactly the same scheme to prove the right inequality of \eqref{soussur8}, namely we begin by showing that there exists a constant $\overline{\xi}$ such that 
$$\forall \xi \in \R,~~~h(\xi) \leq U(0,\xi+\overline{\xi})+q_0 \big[\Phi_0(0)\chi( \xi+\overline{\xi})+\Phi_1(0)\big(1-\chi( \xi+\overline{\xi})\big)\big].$$
Then we can show that there exists a positive constant $\overline{\omega}$ ($\overline{\omega}$ could actually be equal to $\underline{\omega}$ without loss of generality) such that if we take $\overline{\Lambda}(t)=-\overline{\omega}q_0(1-e^{-\mu t}) + \overline{\xi}$ (we need here that $\overline{\Lambda}' \leq 0$), and if we define
$$\overline{u}(t,\xi)=\min \Big\{1, U\big(t,\xi+\overline{\Lambda}(t)\big) +q_0e^{-\mu t}\Big[ \chi\big(\xi+\overline{\Lambda}(t)\big)\Phi_0(t)+\big\{1-\chi\big(\xi+\overline{\Lambda}(t)\big)\big\}\Phi_1(t)\Big]\Big\},$$
then, for all $(t,\xi)\in (0,+\infty)\times \R$ such that $\overline{u}(t,\xi)<1$ we have
$$\overline{u}_t(t,\xi)- c\overline{u}_{\xi}(t,\xi)-\overline{u}_{\xi\xi}(t,\xi)-f^T\big(t,\overline{u}(t,\xi)\big)\geq 0,$$
If we define the real number $\overline{\xi}^0=\overline{\xi}-\overline{\omega}q_0$, we conclude that for all $t\geq0$ and $\xi \in \R$, we have
\begin{align*}
v(t,\xi) & \leq   U\big(t,\xi+\overline{\Lambda}(t)\big)+q_0e^{-\mu t}\big[\chi\big(\xi+\overline{\Lambda}(t)\big)\Phi_0(t)+\big\{1-\chi\big(\xi+\overline{\Lambda}(t)\big)\big\} \Phi_1(t)\big]\\
&\leq  U(t,\xi+\overline{\xi}^0)+q_0\big(\|\Phi_0\|_\infty+\|\Phi_1\|_\infty\big)e^{-\mu t}.
\end{align*}
And the proof of Lemma \ref{lemstab18} is complete.
\end{proof}
\begin{lemma}
\label{lemstab28}
Let $\gamma \in (0,1)$ be as in Lemma \ref{lemstab18}. There exists a positive real number $D$ such that if for some constant $\xi^\sharp$ and some $0<\e<\gamma$, we have
$$\forall \xi \in \R,~~|v(0,\xi)-U(0,\xi+\xi^\sharp)|\leq \e,$$
then, 
$$\forall t \in \R^+,~~\forall \xi \in \R,~~|v(t,\xi)-U(t,\xi+\xi^\sharp)|\leq D\e.$$
\end{lemma}
\begin{proof}
We can adapt the previous proof. If we take $q_0=\e$, then, , for all $\xi\in\R$, we have
$$U(0,\xi+\xi^\sharp)-q_0 \leq h(\xi)=v(0,\xi) \leq U(0,\xi+\xi^\sharp)+q_0.$$
We can then choose $\underline{\xi}$ and $\overline{\xi}$ equal to $\xi^\sharp$. 
Consequently, if we denote $D=\|U_\xi\|_\infty \underline{\omega}+\|\Phi_0\|_\infty+\|\Phi_1\|_\infty$ (independent of $\e$), the conclusion of Lemma \ref{lemstab18} with $\underline{\xi}^0=\underline{\xi}+\underline{\omega}q_0=\xi^\sharp+\underline{\omega}\e$ and $\overline{\xi}^0=\overline{\xi}-\overline{\omega}q_0=\xi^\sharp-\overline{\omega}\e$ becomes
$$\forall t \in \R^+,~~\forall \xi \in \R,~U(t,\xi+\xi^\sharp)-D\e   \leq v(t,\xi) \leq U(t,\xi+\xi^\sharp)+D\e ,$$
The proof of Lemma \ref{lemstab28} is complete.
\end{proof}
\noindent
Before carrying out the proof of Theorem \ref{stability}, we need an additional Liouville type lemma for the solution which are trapped between two shifts of a front. 
\begin{lemma}
\label{lemstab38}
Let $v \in \mathcal{C}^{1+\frac{\alpha}{2},2+\alpha}(\R \times \R,(0,1))$ (with $0<\alpha<1$) be a solution of $\partial_t v -c \partial_\xi v - \partial_{\xi\xi} v =f^T(t,v)$ on $\R^2$ such that 
\begin{equation}
\label{SSTT38}
\forall (t,\xi)\in  \R^2,~~U(t,\xi) \leq v(t,\xi) \leq U(t,\xi-a).
\end{equation}
where $(U(t,\xi),c)$ is a pulsating front solution of \eqref{travel}, and $a$ is a nonnegative real number. Then there exists $b \in [0,a]$ such that
$$\forall (t,\xi) \in \R^2,~~v(t,\xi)=U(t,\xi-b).$$
\end{lemma}
\begin{proof}
According to \eqref{SSTT38}, we have that $v(\cdot,-\infty)=1$ and $v(\cdot,+\infty)=0$ uniformly on $\R$. We are going to show that $v(\cdot,\xi)$ is a $T$-periodic function for any real $\xi$. Let $\e \in \{-1,1\}$. As $v(\cdot,+\infty)=0$ uniformly on $\R$, there exists a real number $R_+$ such that 
$$\forall t \in \R,~~\forall \xi \in [R_+,+\infty),~~v(t,\xi) \leq \delta_+,$$
where $\delta_+\in (0,1)$ is defined in Lemma \ref{TH1comparaison}, with $\g=\G=f^T$.\\
As $v(\cdot,-\infty)=1$ uniformly on $\R$, there exists  a real $\sigma_\e$ such that
$$ \forall t \in \R,~~\forall \xi \in (-\infty,R_+],~~v(t+\e T,\xi-\sigma_\e) \geq 1-\delta_-,$$
where $\delta_-\in (0,1)$ is defined in Lemma \ref{TH2comparaison}, with $\g=\G=f^T$.\\
We apply the comparison principles of Lemma \ref{TH1comparaison} on $\R \times [R_+,+\infty)$ and of Lemma \ref{TH2comparaison} on $\R \times (-\infty,R_+]$ to the functions $\v(t,\xi)=v(t,\xi)$ and $\V(t,\xi)=v(t+\e T,\xi-\sigma_\e)$ 
\begin{equation}
\label{juin}
\forall t \in \R,~~\forall \xi \in \R,~~v(t,\xi) \leq v(t+\e T,\xi-\sigma_\e).
\end{equation}
We define 
$$\sigma_\e^*:=\inf \{\sigma_\e \in \R~|~\forall t \in \R ,~~\forall \xi \in \R,~ v(t,\xi)\leq v(t+\e T,\xi -\sigma_\e)  \},$$
which is a well defined real number such that $\sigma_\e^*\leq\sigma_\e$ by \eqref{juin}.
We have by continuity 
\begin{equation}
\label{coince18}
\forall t \in \R,~~\forall \xi \in \R,~~v(t,\xi)\leq v(t+\e T,\xi -\sigma_\e^*).
\end{equation}
As $v(\cdot,-\infty)=1$ uniformly on $\R$, there exists $R_-<R_+$ such that 
$$\forall \sigma \geq \sigma_\e^*-1,~~\forall t \in \R,~~\forall \xi \in (-\infty,R_-],~v(t+\e T,\xi-\sigma)\geq 1-\delta_-,$$
where $\delta_- \in (0,1)$ is as above.\\
We define $$\eta:=\inf\limits_{(t,\xi)\in \R \times [R_-,R_+]}\{v(t+\e T,\xi-\sigma_\e^*)-v(t,\xi)\}.$$
Two cases can occur, either $\eta$ is positive, or $\eta$ is equal to zero. If $\eta>0$, since $\partial_\xi v$ is globally bounded in $\R \times \R$ by standard parabolic estimates, there exists  $\overline{\sigma}_\e \in(\sigma_\e^*-1,\sigma_\e^*)$ such that 
$$\forall t \in \R,~ \forall \xi \in [R_-,R_+],~~v(t,\xi)<v(t+\e T,\xi-\overline{\sigma}_\e).$$
We can apply Lemma \ref{TH1comparaison} to the functions $v$ and $v(\cdot+\e T,\cdot-\overline{\sigma}_\e)$ on $\R \times [R_+,+\infty)$. We obtain 
$$\forall t \in \R,~~ \forall \xi \in [R_+,+\infty),~~v(t,\xi) \leq v(t+\e T,\xi-\overline{\sigma}_\e).$$
In the same way, we can apply Lemma \ref{TH2comparaison} to the functions $v$ and $v(\cdot+\e T,\cdot-\overline{\sigma}_\e)$ on $\R \times (-\infty,R_-]$. We obtain 
$$\forall t \in \R,~~\forall \xi \in (-\infty,R_-],~~v(t,\xi) \leq v(t+\e T,\xi-\overline{\sigma}_\e).$$
To summarize,
$$\forall t \in \R,~~\forall \xi \in  \R,~~v(t,\xi) \leq v(t+ \e T,\xi-\overline{\sigma}_\e).$$
This contradicts the definition of $\sigma_\e^*$. Consequently, we have $\eta=0$. So, there exists  a sequence $(t_n,\xi_n)_n \subset \R \times [R_-,R_+]$ such that 
\begin{equation}
\label{suiteinf8}
v(t_n+\e T, \xi_n-\sigma_\e^*)-v(t_n,\xi_n) \xrightarrow{n \to +\infty} 0.
\end{equation}
We write $t_n=k_nT+t_n'$, with $k_n \in \Z$ and $t_n' \in (0,T]$. The sequences $(t_n')_n$ and $(\xi_n)_n$ are bounded. We thus have up to extraction of a subsequence that $t_n' \xrightarrow{n\to +\infty} t^*\in \R$ and $\xi_n \xrightarrow{n\to +\infty} \xi^*\in \R$. We define $v_n(t,\xi)=v(t+k_nT,\xi)$. As $f^T$ is $T-$periodic, $v_n$ satisfies the same equation as $v$. Furthermore, we have $0 \leq v_n \leq 1$. Consequently, as $(f^T(\cdot,v_n))_n$ is bounded in  $L^\infty(\R^2)$, the parabolic regularity theory implies that for $1\leq p <+\infty$, $(v_n)_n$ is bounded in $W^{1,2;p}_{\text{loc}}(\R^2)$. Yet $W^{1,2;p}_{\text{loc}}(\R^2)$ embeds compactly into $\mathcal{C}^{0,\alpha}_{\text{loc}}(\R^2)$ for $\alpha \in (0,1-2/p)$ for $p>2$. So, there exists $v^*$ such that up to extraction of a subsequence, $(v_n)_n$ converges to $v^*$ in $W^{1,2;p}_{\text{loc}}(\R^2)$ weakly and in $\mathcal{C}^{0,\alpha}_{\text{loc}}(\R^2)$ for all $1< p < +\infty$ and for all $\alpha \in (0,1)$. The function $0\leq v^*\leq 1$ thus satisfies in the sense of distribution the equation $\partial_t v^* -c \partial_\xi v^* - \partial_{\xi\xi} v^* =f^T(t,v^*)$ on $\R^2$. Actually, by parabolic regularity theory, the function $v^*$ is in fact of class $\mathcal{C}^{1,2}(\R^2)$ and it satisfies the previous equation in the classical sense. According to \eqref{suiteinf8}, we thus have 
\begin{equation*}
v^*(t^*+\e T, \xi^*-\sigma_\e^*)=v^*(t^*,\xi^*). 
\end{equation*}
Since on the other hand $v^*(t+\e T, \xi-\sigma_\e^*)\geq v^*(t,\xi)$, for all $(t,\xi) \in \R\times \R$ by \eqref{coince18}, the strong  maximum principle implies that \\
\begin{equation*}
\forall t \in (-\infty, t^*],~~\forall \xi \in \R,~~v^*(t,\xi)=v^*(t+\e T,\xi-\sigma^*_\e).
\end{equation*}
As $t^*\in [0,T]$, the previous equality is true on $\R_-\times \R$, and in particular, we have 
\begin{equation}
\label{coince58}
\forall k \in \N,~v^*(0,0) = v^*(- k T,\e k\sigma_\e^*).
\end{equation}
According to \eqref{SSTT38} and \eqref{coince58} we have that 
\begin{equation}
\label{coince68}
\forall k \in \N,~~U(0,\e k\sigma_\e^*) = U(-k T,\e k\sigma_\e^*)\leq v^*(-k T,\e k\sigma_\e^*)= v^*(0,0) \leq  U(0,-a)<1.
\end{equation}
In the same way, we have that 
\begin{equation}
\label{coince78}
\forall k \in \N,~~0< U(0,0)  \leq v^*(0,0) = v^*(-k T,\e k\sigma_\e^*) \leq U(-k T,\e k\sigma_\e^*-a)=U(0,\e k\sigma_\e^*-a).
\end{equation}
We suppose that $\e\sigma_\e^*<0$. So, we have $\e k\sigma_\e^* \xrightarrow{k \to +\infty} -\infty$. Consequently, as $U(0,-\infty)=1$, we obtain a contradiction in \eqref{coince68} if we pass in the limit when $k$ tends to $+ \infty$.\\
We suppose that $\e\sigma_\e^*>0$. So, we have $\e k\sigma_\e^* \xrightarrow{k \to +\infty} +\infty$. Consequently, as $U(0,+\infty)=0$, we obtain a contradiction in \eqref{coince78} if we pass in the limit when $k$ tends to $+ \infty$.
\vspace{0.5em}\\
Consequently, we have $\sigma_\e^*=0$, and the equation \eqref{coince18} rewrites 
\begin{equation*}
\forall t \in \R,~~\forall \xi \in \R,~~v(t,\xi)\leq v(t+\e T,\xi).
\end{equation*}
For $\e=1$, the previous inequality rewrites 
\begin{equation}
\label{coince88}
\forall t \in \R,~~\forall \xi \in \R,~~v(t,\xi)\leq v(t+ T,\xi).
\end{equation}
And, for $\e=-1$, we have $v(t,\xi)\leq v(t- T,\xi)$, for all $(t,\xi) \in \R^2$.
If we define $t'=t-T$, we have 
\begin{equation}
\label{coince98}
\forall t' \in \R,~~\forall \xi \in \R,~~v(t'+T,\xi)\leq v(t',\xi).
\end{equation}
According to \eqref{coince88} and \eqref{coince98}, we get that 
$$\forall t \in \R,~~\forall \xi \in \R,v(t,\xi) = v(t+T,\xi).$$
which is the desired conclusion.
\end{proof}
\noindent
We can now prove Theorem \ref{stability}.
\begin{proof}
Let $\gamma \in (0,1)$ be as in Lemma \ref{lemstab18} and let $u$ be as in the statement of Theorem \ref{stability}.
We define $v_n(t,\xi)=v(t+nT,\xi)$. Up to extraction of a subsequence, $(v_n)_n$ converges locally uniformly to a $\mathcal{C}^{1+\frac{\alpha}{2},2+\alpha}(\R^2)$ solution $v_\infty$ of $\partial_t v_\infty -c \partial_\xi v_\infty - \partial_{\xi\xi} v_\infty =f^T(t,v_\infty)$, for $\alpha \in (0,1)$. According to Lemma \ref{lemstab18}, there exist some real numbers $\underline{\xi}^0$ and $\overline{\xi}^0$ such that 
$$\forall t \in \R,~~\forall \xi \in \R,~~U(t,\xi+\underline{\xi}^0) \leq v_\infty(t,\xi) \leq U(t,\xi+\overline{\xi}^0).$$
Consequently, Lemma \ref{lemstab38} implies that there exists $\xi_0$ between $\underline{\xi}^0$ and $\overline{\xi}^0$ such that 
$$\forall \xi \in \R,~~\forall t\in \R,~~v_\infty(t,\xi)=U(t,\xi-\xi_0).$$
Let $\e\in (0,\gamma)$ be fixed.  According to the fact that $U(\cdot,-\infty)=1$ and $U(\cdot,+\infty)=0$, and according to Lemma \ref{lemstab18}, there exists an integer $n_0$ such that 
$$\forall \xi \in \R,~~|v(n_0T,\xi)-U(0,\xi-\xi_0)|\leq \e.$$
Lemma \ref{lemstab28} yields
$$\forall t\in \R^+,~~\forall \xi \in \R,~~|v(n_0T+t,\xi)-U(t,\xi-\xi_0)|\leq D\e.$$
where $D$ is independent of $\e$. Since $\e>0$ could be arbitrary small, we get that
$$\lim\limits_{t \to +\infty}v(t,\xi)-U(t,\xi-\xi_0)=0,~~~~\text{uniformly on}~\R.$$
That concludes the proof of the theorem.
\end{proof}
\noindent
Once the global stability of pulsating fronts is established, the uniqueness result in Theorem \ref{existenceunicite} is an easy corollary. This method was used in \cite{Shen3}. In the present paper, we preferred to prove first the uniqueness in Section \ref{caracterisation} because the proof uses some new comparison principle (which have their own interest) and which lead to the monotonicity result.
\begin{proof}
One has to show that if $(U_1,c_1)$ and $(U_2,c_2)$ are two pulsating fronts solving \eqref{travel}, then $c_1=c_2$ and $U_1$ and $U_2$ are equal up to shift in space. Theorem \ref{stability} yields the existence of $\xi_0\in \R$ such that 
\begin{equation}
\label{stabunic8}
\sup_{\xi \in \R}\big|U_2\big(t,\xi+(c_1-c_2)t\big)-U_1\big(t,\xi+\xi_0\big)\big|\xrightarrow{t \to +\infty}0.
\end{equation}
Let $k \in \Z$, and $(t,\xi) \in \R^2$. By $T$-periodicity of $U_1$ and $U_2$ we have 
\begin{equation}
\label{stabunic18}
\big|U_2\big(t+kT,\xi+(c_1-c_2)(t+kT)\big)-U_1\big(t+kT,\xi+\xi_0\big)\big|=\big|U_2\big(t,\xi+(c_1-c_2)(t+kT)\big)-U_1\big(t,\xi+\xi_0\big)\big|.
\end{equation}
If $c_1 \neq c_2$, we pass to the limit when $k \to +\infty$. According to \eqref{stabunic8}, the first term of \eqref{stabunic18} converges to zero, whereas the second term converges to $1-U_1(t,\xi+\xi_0)>0$ if $c_1<c_2$, and to $U_1(t,\xi+\xi_0)>0$ if $c_1>c_2$. Consequently we have $c_1=c_2$, and \eqref{stabunic18} becomes 
$$\big|U_2\big(t+kT,\xi\big)-U_1\big(t+kT,\xi+\xi_0\big)\big|=\big|U_2\big(t,\xi\big)-U_1\big(t,\xi+\xi_0\big)\big|.$$
By passing to the limit when $k \to +\infty$, we get that $U_2(t,\xi)=U_1(t,\xi+\xi_0)$. Hence, the proof of the uniqueness result in Theorem \ref{stability} is complete.
\end{proof}
\section{Pulsating fronts for nonlinearities of small periods}
\label{sectionpetite}
In this section, we focus on the dependance on $T$ for the pulsating fronts solving
$$\partial_tu-\partial_{xx}u=f^{T}(t,u),~~\forall t \in \R,~~\forall x \in \R.$$
where $f^{T}(t,u)=f(\frac{t}{T},u)$, with $f(\cdot,u)$ a $1$-periodic function.\\ 
Let $\alpha \in [0,1]$. We denote $w_T(\alpha,\cdot)$ the solution of the Cauchy problem 
$$\begin{cases}
y'=f^{T}(t,y)~\text{on}~\R,\\
y(0)=\alpha.
\end{cases}$$
We remind the reader that the $\text{Poincar\'e}$ map associated with $f^{T}$ is the function $P_T:[0,1] \to [0,1]$ defined by 
$$P_T(\alpha)=w_T(\alpha,T).$$
\subsection{Existence and uniqueness}
This subsection is devoted to proving Theorem \ref{petiteex}. We are going to show that for $T>0$ small enough, the $\text{Poincar\'e}$ map associated with the function $f^{T}$ admits exactly one unstable fixed point which is strictly included between $0$ and $1$. Then, by \cite{ABC}, there exists a pulsating front solving \eqref{travel}. We begin by proving the existence of a fixed point of the $\text{Poincar\'e}$ map between $0$ and $1$.
\begin{lemma}
\label{pr1}
Let $T>0$.There are solutions of the problem 
\begin{equation}
\label{problem1}
\begin{cases}
y'=f^{T}(t,y)~~\text{on}~\R,\\
0<y<1~~\text{on}~\R,\\
y(0)=y(T).
\end{cases}
\end{equation}
\end{lemma}
\begin{proof}
We define $\Phi_T(\alpha)=P_T(\alpha)-\alpha$. By hypotheses, $0$ and $1$ are stable fixed points of $P_T$. So, $$\begin{cases}\Phi_T(0)=0,\\ \Phi_T'(0)<0, \end{cases}~\text{and}~\begin{cases}\Phi_T(1)=0,\\ \Phi_T'(1)<0.\end{cases}$$\\
By continuity, there exists $\alpha_T \in (0,1)$ such that $\Phi_T(\alpha_T)=0$, that is $P_T(\alpha_T)=\alpha_T$. So, the function $w_T(\alpha_T,\cdot)$ satisfies Problem \eqref{problem1} as a consequence of the Cauchy-Lipschitz theorem.
\end{proof}
\noindent
Let $t \mapsto \theta_T(t)$ be a solution of Problem \eqref{problem1}. We remind that $\theta_T(0)$ is a fixed point of the $\text{Poincar\'e}$ map associated with $f^{T}$. It is unstable if $P_T'(\theta(0))>1$. We saw in Section \ref{characterization} that this condition is equivalent to the fact that the principal eigenvalue associated with $f^{T}$ and $\theta_T$ is negative. We are thus interested in the sign of 
\begin{equation}
\label{interversion}
\lambda_{\theta_T,f^{T}}=-\frac{1}{T}\int_0^{T} f_u^{T}(s,\theta_T(s)) ds = -\int_0^1 f_u(t,\theta_T(tT))dt.
\end{equation}
We consider a sequence  of positive real numbers $(T_n)_n$ such that $T_n \xrightarrow{n\to+\infty}0$.
\begin{lemma}
\label{convunif}
The sequence $(\theta_{T_n})_n$ converges up to extraction of a subsequence uniformly on $\R$ to a constant function $\theta_0$.
\end{lemma}
\begin{proof}
Let $K$ be a positive constant such that 
$$|\theta_{T_n}'(t)| \leq K,~~\forall t \in \R,~~\forall n\in \N.$$
Let $t \in \R$ and $n\in\N$. There exists an integer $k_n$ such that $t \in [k_nT_n,(k_n+1)T_n)$. As $\theta_n(0)\in (0,1)$, there exists a real number $\theta_0$ such that up to extraction of a subsequence 
$$\theta_{T_n}(0)=\theta_{T_n}(k_nT_n) \xrightarrow{n \to \infty} \theta_0.$$
Yet, by the mean value theorem, we get that 
$$|\theta_{T_n}(t)-\theta_{T_n}(k_nT_n)| \leq K|t-k_nT_n| \leq KT_n,~~\forall t \in \R,~~\forall n\in \N.$$
Since $T_n \xrightarrow{n \to \infty}0$, for all $t\in \R$, we have 
$$|\theta_{T_n}(t)-\theta_0| \leq  |\theta_{T_n}(t)-\theta_{T_n}(k_nT_n)|+  |\theta_0-\theta_{T_n}(k_nT_n)| \leq KT_n +|\theta_0-\theta_{T_n}(k_nT_n)|\xrightarrow{n \to \infty} 0,$$
and the proof of Lemma \ref{convunif} is complete.
\end{proof}
\begin{lemma}
\label{signjohn}
Up to extraction of a subsequence, we have 
$$\lambda_{\theta_{T_n},f^{T_n}} \xrightarrow{n \to \infty} -\int_0^1 f_u(s,\theta_0)ds.$$
\end{lemma}
\begin{proof}
Let us note that $(\theta_{T_n}(\cdot T_n))_n$ converges up to extraction of a subsequence uniformly on $\R$ to $\theta_0$. Indeed, for any $\epsilon>0$, there exists $n_0\in\N$ such that 
$$|\theta_{T_n}(s)-\theta_0|\leq \epsilon,~~\forall n\geq n_0, ~~\forall s \in \R.$$
So,
$$|\theta_{T_n}(tT_n)-\theta_0|\leq \epsilon,~~\forall n\geq n_0,~~\forall t \in \R.$$
We can then move the limit inside the integral \eqref{interversion}.
\end{proof}
\noindent
Consequently, for $n$ large enough (ie $T_n$ small enough), $\lambda_{\theta_n,f^{T_n}}$ has the same sign as 
$$h(\theta_0):=-\int_0^1 f_u(s,\theta_0)ds.$$
The function $h$ satisfies $h=-g'$, where $g$ is the function defined in \eqref{fonctiong}.
\begin{lemma}
The real number $\theta_0$ satisfies 
$$g(\theta_0)=0.$$
\end{lemma}
\begin{proof}
We have $$\theta_{T_n}'(t)=f^{T_n}(t,\theta_{T_n}(t)),~~\forall n \in \N,~~\forall t \in \R.$$
We integrate this equation between $0$ and $T_n$ 
$$\theta_{T_n}(T_n)-\theta_n(0)=\int_0^{T_n} f^{T_n}(s,\theta_{T_n}(s))ds,~~\forall n \in \N.$$
Consequently $$\int_0^{T_n} f^{T_n}(s,\theta_{T_n}(s))ds=0,~~\forall n \in \N.$$
By changing the variable $t=\frac{s}{T_n}$, we get that $$\int_0^1 f(t,\theta_{T_n}(tT_n))dt=0,~~\forall n \in \N.$$
Finally, we pass to the limit and we obtain $\int_0^1 f(t,\theta_0)dt=0.$
\end{proof}
\noindent
So $g(\theta_0)=0$. Consequently, $\theta_0$ is equal to $0$, $\theta_g$ or $1$ . We are now going to justify that $\theta_0$ is different of $0$ et $1$.
\begin{lemma}
We have the following equality 
$$\theta_0=\theta_g.$$
\end{lemma}
\begin{proof}
Let us begin by showing that $\theta_0$ is different from $0$. We argue bwoc, supposing that $\theta_0=0$. So $U_n(t):=\theta_{T_n}(tT_n)$ converges uniformly to the null function on $\R$ and we have $U_n'=T_nf(t,U_n)$, on $\R$. We divide this equation by $T_nU_n$, then we integrate between $0$ and $1$ 
$$\frac{1}{T_n}\int_0^1 \frac{U_n'(t)}{U_n(t)}dt=\int_0^1 \frac{f(t,U_n(t))}{U_n(t)}dt,~~\forall n \in \N.$$ So 
$$\frac{1}{T_n}[\log(U_n(1))-\log(U_n(0))]=\int_0^1 \frac{f(t,U_n(t))-f(t,0)}{U_n(t)-0}dt,~~\forall n \in \N.$$ Consequently 
$$\int_0^1 \frac{f(t,U_n(t))-f(t,0)}{U_n(t)-0}dt=0,~~\forall n \in \N.$$ We pass to the limit when $n$ converges to infinity 
$$\int_0^1 f_u(t,0)dt=0.$$ which contradicts the hypothesis \eqref{bismoyT}. Consequently $\theta_0\neq 0$. To show that $\theta_0$ is different from $1$, the proof is similar but one has to divide by $T_n ( U_n-1)$ instead of $T_nU_n$.
\end{proof}
\noindent
The previous lemma implies that $\theta_g$ is the unique accumulation point of the sequence $(\theta_{T_n})_n$. Consequently, the convergences in Lemma \ref{convunif} and in Lemma \ref{signjohn} are not up to extraction of a subsequence, and we have
$$\lim\limits_{T_n\to0^+}\lambda_{\theta_{T_n},f^{T_n}}=-g'(\theta_g).$$
Actually, we have even
\begin{equation}
\label{Tpetit}
\lim\limits_{T\to0^+}\lambda_{\theta_{T},f^{T}}=-g'(\theta_g).
\end{equation}
 According to \eqref{sdgtz} and \eqref{Tpetit}, we can define $T_f>0$ such that $\lambda_{\theta_T,f^T}<0$ for all $T\in(0,T_f)$ and for all $\theta_T$ solving \eqref{problem1}. In other words, for all $T\in (0,T_f)$, if a the function $t \mapsto \theta_T(t)$ solves \eqref{problem1}, then it is an unstable equilibrium state and $\theta_T(0)$ is an unstable fixed point of the $\text{Poincar\'e}$ map associated with $f^{T}$. To finish the proof of Theorem \ref{petiteex}, we are going to show the uniqueness of the fixed point in $(0,1),$ for $T\in (0,T_f)$.
\begin{lemma}
Let $\theta_T$ and $\Psi_T$ be two solutions of Problem \eqref{problem1}. For all $T\in (0,T_f)$, we have 
 $$\theta_T(0)=\Psi_T(0)~(\text{ie}~\theta_T \equiv \Psi_T).$$
\end{lemma}
\begin{proof}
Let $T\in(0,T_f)$. We define $\Phi_T(\alpha)=P_T(\alpha)-\alpha$. Let us suppose bwoc that $\theta_T(0) \neq \Psi_T(0)$.We saw that 
$$\begin{cases}
\Phi_T(\theta_T(0))=0,\\
\Phi_T'(\theta_T(0))>0,
\end{cases}
 ~\text{and}~~
 \begin{cases}
 \Phi_T(\Psi_T(0))=0,\\
  \Phi_T'(\Psi_T(0))>0.
 \end{cases}$$
Necessarily, there exists $\alpha_T$ between $\theta_T(0)$ and $\Psi_T(0)$ such that $\Phi_T(\alpha_T)=0~\text{and}~\Phi'_T(\alpha_T)\leq0$. Consequently, $P_T(\alpha_T)=\alpha_T~\text{and}~P'_T(\alpha_T)\leq1$. This contradicts the fact that all fixed points of the $\text{Poincar\'e}$ map in $(0,1)$ are unstable. 
\end{proof}\noindent
To summarize, for all $T\in(0,T_f)$, the $\text{Poincar\'e}$ map associated with $f^T$ has a unique fixed point $\theta_T(0)$ between $0$ and $1$, where $\theta_T$ is the unique solution of Problem \eqref{problem1}. Furthermore, $\theta_T(0)\xrightarrow{T\to0}\theta_g$. 
\subsection{Convergence of the couple $(c_T,U_T)$ as $T \to 0$.}
We here prove Theorem \ref{petiteconv}. Let $T\in(0,T_f)$ and $M>0$. In \cite{ABC}, the couple $(c_T, U_T)$ is built as limit when $M$ tends to infinity of the couple $(c^M_T, U^M_T)$ solving 
$$\begin{cases}
(U^M_T)_t-c^M_T(U^M_T)_{\xi}-(U^M_T)_{\xi \xi}-f^{T}(t,U^M_T)=0,~~~~\text{on}~\R \times (-M,M), \\
U^M_T(t,-M)=1,~~U^M_T(t,+M)=0,~~~~\forall t \in \R,\\
U^M_T(T,\xi)=U^M_T(0,\xi),~~~\forall \xi \in [-M,M], \\
U^M_T(0,0)=\theta_g.
\end{cases}$$
We give a lemma which comes again from the article of Alikakos, Bates and Chen \cite{ABC} which is going to serve us to bound the speeds $c_T$.
\begin{lemma}{\cite{ABC}}
\label{lemABC}
Let $M>1$ be a fixed constant. If $(\tilde{U},\tilde{c})$ satisfies \\
$$\begin{cases}
\tilde{U}_t-\tilde{c} \tilde{U}_{\xi}-\tilde{U}_{\xi \xi}-f^{T}(t,\tilde{U})\leq0,~~~~\text{on}~\R \times (-M,M), \\
\tilde{U}(t,-M)\leq 1,~~\tilde{U}(t,+M)\leq 0,~~~~\forall t \in \R,\\
\tilde{U}(T,\xi)\leq\tilde{U}(0,\xi),~~~~\forall \xi \in [-M,M], \\
\tilde{U}(0,0)\geq \theta_g,
\end{cases}$$
then, we have $$c_T^M\leq\tilde{c}.$$
\end{lemma}
\noindent
As $f$ is of class $\mathcal{C}^1(\R\times\R,\R)$, the function $f_u$ is bounded on $[0,1]\times [0,1]$.
Consequently, there exists a function $\overline{f}$ KPP and a function $\underline{f}$ the opposite of which is KPP such that 
\begin{equation*}
\underline{f}(u)\leq f(t,u) \leq \overline{f}(u),~~\forall (t,u)\in [0,1]\times[0,1].
\end{equation*}
There are planar fronts $(\underline{c},\underline{U})$ and $(\overline{c},\overline{U})$ solving 
$$\begin{cases} 
\underline{U}''+\underline{c}~\underline{U}'+\underline{f}(\underline{U})=0~~~\text{on}~\R,\\
\underline{U}(-\infty)=1,~~\underline{U}(+\infty)=0,\\
\underline{U}(0)=\theta_g,
\end{cases}~~\text{and}~~
\begin{cases} 
\overline{U}''+\overline{c}~\overline{U}'+\overline{f}(\overline{U})=0~~~\text{on}~\R,\\
\overline{U}(-\infty)=1,~~\overline{U}(+\infty)=0,\\
\overline{U}(0)=\theta_g.
\end{cases}$$
We can again obtain these couples from the limits when $M$ tends to infinity of the couples $(\underline{c}^M,\underline{U}^M)$ and $(\overline{c}^M,\overline{U}^M)$ solving 
$$ \begin{cases} 
(\underline{U}^{M}){''} +\underline{c}^M~(\underline{U}^{M}){'}+\underline{f}
(\underline{U}^M)=0~\text{on}~(-M,M),\\
\underline{U}^M(-M)=1,~~\underline{U}^M(+M)=0,\\
\underline{U}^M(0)=\theta_g,
\end{cases}$$
$$\begin{cases} 
(\overline{U}^{M}){''}+\overline{c}^M~(\overline{U}^{M}){'}+\overline{f}(\overline{U}^M)=0~\text{on}~(-M,M),\\
\overline{U}^M(-M)=1,~~\overline{U}^M(+M)=0,\\
\overline{U}^M(0)=\theta_g.
\end{cases}$$
For the detail of this construction, we can refer to the article of Berestycki and Chapuisat  \cite{BC}. The next proposition supplies bounds for the speed  $c_T$.
\begin{proposition}
Let $T\in(0,T_f)$. The speed $c_T$ satisfies 
\begin{equation*}
\overline{c}\leq c_T \leq \underline{c}.
\end{equation*}
\end{proposition}
\begin{proof}
Let $M>1$ and $T\in (0,T_f)$. We are going to twice apply Lemma \ref{lemABC}. We have on $(-M,M)$
$$(\underline{U}^{M})_{t}-(\underline{U}^{M})_{\xi \xi} -\underline{c}^M~(\underline{U}^M)_\xi-f^{T}(t,\underline{U}^M)=\underline{f}(\underline{U}^M)-f^{T}(t,\underline{U}^M)\leq0.$$
Furthermore, $\underline{U}^M(-M)=1,~~\underline{U}^M(+M)=0~~\text{and}~~\underline{U}^M(0)=\theta_g.$
Consequently 
$$c_T^M\leq \underline{c}^M.$$
We define the functions $\tilde{f}^T:\R\times[0,1]\to\R,~(t,u)\mapsto-f^T(t,1-u)$ and $V_T^M:\R\times[-M,M]\to\R,~(t,\xi)\mapsto 1-U(t,-\xi)$. The couple $(-c_T^M,V^M_T)$ solves the problem
$$ \begin{cases} 
(V^M_T)_t-(-c^M_T)(V^M_T)_{\xi}-(V^M_T)_{\xi \xi}-\tilde{f}^{T}(t,V^M_T)=0,~~~~\text{on}~\R \times (-M,M), \\
V^M_T(t,-M)=1,~~V^M_T(t,+M)=0,~~~~\forall t \in \R,\\
V^M_T(T,\xi)=V^M_T(0,\xi),~~~\forall \xi \in [-M,M], \\
V^M_T(0,0)=1-\theta_g.
\end{cases}$$
Since the nonlinearity $\tilde{f}^T$ satisfies the hypotheses of \cite{ABC}, Lemma \ref{lemABC} could be rewrites:\\
\textit{
Let $M>1$ be a fixed constant. If $(\tilde{U},\tilde{c})$ satisfies \\
$$\begin{cases}
\tilde{U}_t-\tilde{c} \tilde{U}_{\xi}-\tilde{U}_{\xi \xi}-\tilde{f}^{T}(t,\tilde{U})\leq0,~~~~\text{on}~\R \times (-M,M), \\
\tilde{U}(t,-M)\leq 1,~~\tilde{U}(t,+M)\leq 0,~~~~\forall t \in \R,\\
\tilde{U}(T,\xi)\leq\tilde{U}(0,\xi),~~~~\forall \xi \in [-M,M], \\
\tilde{U}(0,0)\geq 1-\theta_g,
\end{cases}$$
then, we have $$-c_T^M\leq\tilde{c}.$$}
Yet, defining $\overline{V}^{M}(\xi)=1-\overline{U}^{M}(-\xi)$, we have for all $\xi \in (-M,M)$
\begin{equation*}
\begin{split}
&\overline{V}^{M}_{t}(\xi)-\overline{V}^{M}_{\xi \xi}(\xi) -(-\overline{c}^M)~\overline{V}^M_\xi(\xi)-\tilde{f}^{T}(t,\overline{V}^M(\xi))\\
=&-\big\{\overline{U}^{M}_t(-\xi)-\overline{U}^{M}_{\xi\xi}(-\xi)-\overline{c}^M\overline{U}^{M}_\xi(-\xi)\big\}+f^T(t,\overline{U}^M(-\xi))\\
=&~f^{T}(t,\overline{U}^M(-\xi))-\overline{f}(\overline{U}^M(-\xi))\leq0.
\end{split}
\end{equation*}
Furthermore, $\overline{V}^M(-M)=1,~~\overline{V}^M(+M)=0~~\text{and}~~\overline{V}^M(0)=1-\theta_g.$
Consequently 
$$ -c^M_T\leq -\overline{c}^M.$$
We thus showed that 
$$\overline{c}^M \leq c_T^M\leq \underline{c}^M,~~\forall M>1.$$
We pass to the limit when $M$ tends to infinity, it occurs that $\overline{c} \leq c_T\leq \underline{c}.$
\end{proof}
\noindent
Consequently, there exists a sequence $(T_n)_n$, with $T_n\xrightarrow{n\to+\infty}0$ such that $(c_{T_n})_n$ converges to a constant $c^*\in \R$.
\begin{proposition}
\label{convergence}
The sequence $(U_{T_n})_n$ converges up to extraction of a subsequence in $W^{1,2;p}_{\text{loc}}(\R^2)$ weakly and in  $\mathcal{C}^{0,\alpha}_{\text{loc}}(\R^2)$ to a function $U^*$ for any $1< p<+\infty$ and for any $\alpha \in (0,1)$.
\end{proposition}
\begin{proof}
$U_{T_n}$ is solution of $(U_{T_n})_t-c_{T_n}(U_{T_n})_{\xi}-(U_{T_n})_{\xi \xi}-f^{T_n}(t,U_{T_n})=0$ on $\R^2$ and satisfies $0\leq U_{T_n} \leq 1$. Consequently, as $(c_{T_n})_n$ is bounded in $\R$ and as $(f^{T_n}(\cdot,U_{T_n}))_n$ is bounded in  $L^\infty_{\text{loc}}(\R^2)$, the parabolic regularity theory implies that for $1\leq p <+\infty$, $(U_{T_n})_n$ is bounded in $W^{1,2;p}_{\text{loc}}(\R^2)$. Yet $W^{1,2;p}_{\text{loc}}(\R^2)$ embeds compactly into $\mathcal{C}^{0,\alpha}_{\text{loc}}(\R^2)$ for $\alpha \in (0,1-2/p)$, for $p>2$. So, there exists $U^*$ such that up to extraction of a subsequence, $(U_{T_n})_n$ converges to $U^*$ in $W^{1,2;p}_{\text{loc}}(\R^2)$ weakly and in $\mathcal{C}^{0,\alpha}_{\text{loc}}(\R^2)$ for any $1< p<+\infty$ and for any $\alpha \in (0,1)$.
\end{proof}
\begin{proposition}
\label{distribution}
We have the following convergence result 
$$f^{T_n}(\cdot,U_{T_n})\xrightarrow{n\to+\infty}g(U^*)~\text{in}~\mathcal{D}'(\R^2)~\text{(up to extraction of a subsequence)}.$$
\end{proposition}
\begin{proof} 
We still denote $(U_{T_n})_n$ the subsequence of $(U_{T_n})_n$ which converges to $U^*$. Let $\phi \in \mathcal{C}_c^\infty(\mathbb{R}\times\mathbb{R})$. There exist $(a,b,\xi_0,\xi_1) \in \mathbb{R}^4$ such that the support of $\phi$ is included in $K=[a,b] \times [\xi_0,\xi_1]$. So we have 
\begin{flalign*}
 &\int_{\xi=-\infty}^{+\infty} \int_{t=-\infty}^{+\infty} \left[f^{T_n}(t,U_{T_n}(t,\xi))-g(U^*(\xi))\right]\phi(t,\xi) ~dt~ d\xi& &\\
 =& \int_{\xi=\xi_0}^{\xi_1} \int_{t=a}^{b} \left[f^{T_n}(t,U_{T_n}(t,\xi))-g(U^*(\xi))\right]\phi(t,\xi) ~dt~d\xi. & &
\end{flalign*}
Let $(k_n,p_n) \in \mathbb{N}^2$ such that $[k_nT_n,p_nT_n] \subset [a,b]$, $|a-k_nT_n| \leq T_n$ et $|b-p_nT_n| \leq T_n$.\\
\vspace{-0.4em}\\
$\int_{\xi_0}^{\xi_1}\int_a^b \left[f^{T_n}(t,U_{T_n}(t,\xi))-g(U^*(\xi))\right]\phi(t,\xi) ~dt~ d\xi=\underbrace{\int_{\xi_0}^{\xi_1}\int_a^{k_nT_n}}_{I_n} +\underbrace{\int_{\xi_0}^{\xi_1}\int_{k_nT_n}^{p_nT_n}}_{J_n}+\underbrace{\int_{\xi_0}^{\xi_1}\int_{p_nT_n}^b}_{K_n}$.\\
\vspace{-0.1em}\\
The integrals $I_n$ et $K_n$ are treated in the same way: 
\begin{flalign*}
I_n & =\int_{\xi_0}^{\xi_1}\int_a^{k_nT_n} \left[f^{T_n}(t,U_{T_n}(t,\xi))-g(U^*(\xi))\right]\phi(t,x) dt~d\xi & & \\
 & =\int_{-\infty}^{+\infty}\int_{-\infty}^{+\infty} 1_{[a,k_nT_n] \times [\xi_0,\xi_1]}(t,\xi) \left[f^{T_n}(t,U_{T_n}(t,\xi))-g(U^*(\xi))\right]\phi(t,\xi) ~dt~d\xi. 
\end{flalign*}
The quantity $\left[f^{T_n}(t,U_{T_n}(t,\xi))-g(U^*(\xi))\right]\phi(t,\xi)$ is bounded independently of $n$, $t$ and $\xi$. So, by Lebesgue's theorem, the integral tends to $0$ when $n$ tend to infinity. So
$$\lim\limits_{n \to +\infty} I_n=\lim\limits_{n \to +\infty} K_n=0.$$
Let us consider $J_n$:
\begin{flalign*}
J_n  &= \int_{\xi_0}^{\xi_1} \int_{k_nT_n}^{p_nT_n} [f^{T_n}(t,U_{T_n}(t,\xi))-g(U^*(\xi))]\phi(t,\xi) ~dt~d\xi ~~~~~(s=\frac{t}{T_n})&&\\ 
 &= T_n \int_{\xi_0}^{\xi_1} \int_{k_n}^{p_n} \left[f(s,U_{T_n}(sT_n,\xi))-g(U^*(\xi))\right]\phi(sT_n,\xi) ~ds~d\xi & &\\
&= T_n \int_{\xi_0}^{\xi_1}  \sum_{q=k_n}^{p_n-1} \int_{q}^{q+1} \left[f(s,U_{T_n}(sT_n,\xi))-g(U^*(\xi))\right]\phi(sT_n,\xi) ~ds~d\xi & &\\
  &= T_n \int_{\xi_0}^{\xi_1}  \sum_{q=k_n}^{p_n-1} \int_{0}^{1} \left[f(s+q,U_{T_n}(sT_n+qT_n,\xi))-g(U^*(\xi))\right]\phi(sT_n+qT_n,\xi) ~ds~d\xi & &\\
&= T_n \int_{\xi_0}^{\xi_1}  \sum_{q=k_n}^{p_n-1} \int_{0}^{1} \left[f(s,U_{T_n}(sT_n,\xi))-g(U^*(\xi))\right]\phi(sT_n+qT_n,\xi) ~ds~d\xi .&&
\end{flalign*}
We split now the integral into two parts.\\
\begin{flalign*}
   J_n  & =\underbrace{T_n \sum_{q=k_n}^{p_n-1} \int_{0}^{1} \int_{\xi_0}^{\xi_1} \left[ f(s,U^*(\xi))-g(U^*(\xi)) \right] \phi(sT_n+qT_n,\xi) ~d\xi~ds}_{J_n^1}\\
 &+ \underbrace{T_n \int_{\xi_0}^{\xi_1}  \sum_{q=k_n}^{p_n-1} \int_{0}^{1} [f(s,U_{T_n}(sT_n,\xi))-f(s,U^*(\xi))]\phi(sT_n+qT_n,\xi) ~ds~d\xi}_{J_n^2}.
 \end{flalign*}
We study $J^2_n$: 
\begin{flalign*}
& T_n \int_{\xi_0}^{\xi_1}  \sum_{q=k_n}^{p_n-1} \int_{0}^{1} [f(s,U_{T_n}(sT_n,\xi))-f(s,U^*(\xi))]\phi(sT_n+qT_n,\xi) ~ds~d\xi \\
\leq ~& T_n(p_n-k_n)\|f_u\|_{\infty,[0,1]^2} \|\phi\|_\infty\int_{0}^{1} |U_{T_n}(sT_n,\xi)-U^*(\xi)|~ds &&\\
\leq  ~& |b-a| \|f_u\|_{\infty,[0,1]^2} \|\phi\|_\infty \|U_{T_n}-U^*\|_{\infty,[0,1]\times [\xi_0,\xi_1]}\xrightarrow{n \to +\infty}0 ~~(\text{according to Prop}~\ref{convergence}). &&
\end{flalign*}
We study $J^1_n$: \\
Let $\epsilon>0$. By uniform continuity of the function $\phi$, there exists $\eta>0$ such that
$$\|(t,\xi)-(s,\xi')\|\leq \eta \Longrightarrow |\phi(t,\xi)-\phi(s,\xi')|\leq\epsilon.$$
For $n$ large enough, we have
$$\forall q \in [k_n,p_n-1],~~\forall s \in [0,1],~~\forall \xi \in [\xi_0,\xi_1],~~~~\|(sT_n+qT_n,\xi)-(qT_n,\xi)\| \leq \eta.$$
Consequently, we have $|\phi(sT_n+qT_n,\xi)-\phi(qT_n,\xi)|\leq \epsilon$. We write then\\
\begin{flalign*}
J_n^1 &= \underbrace{T_n \int_{\xi_0}^{\xi_1}  \sum_{q=k_n}^{p_n-1} \int_{0}^{1} \Big[f(s,U^*(\xi))-g(U^*(\xi))\Big]\Big[\phi(sT_n+qT_n,\xi)- \phi(qT_n,\xi)\Big]~ds~d\xi }_{J^3_n}\\
&+\underbrace{T_n \int_{\xi_0}^{\xi_1}  \sum_{q=k_n}^{p_n-1} \int_{0}^{1} \Big[f(s,U^*(\xi))-g(U^*(\xi))\Big] \phi(qT_n,\xi)~ds~d\xi}_{J^4_n}.
\end{flalign*}
We start with $J^4_n$:\\
$J^4_n=T_n \int_{\xi_0}^{\xi_1}  \sum_{q=k_n}^{p_n-1} \phi(qT_n,\xi) \left\{  \int_{0}^{1} \Big[f(s,U^*(\xi))-g(U^*(\xi))\Big] ~ds \right\}~d\xi$.\\
By definition of $g$, we have $\int_{0}^{1} \Big[f(s,U^*(\xi))-g(U^*(\xi))\Big] ~ds=0$. Consequently $J^4_n=0$.
\vspace{0.5em}\\
We look now at $J^3_n$:
\begin{flalign*}
|J^3_n| &\leq T_n \int_{\xi_0}^{\xi_1}  \sum_{q=k_n}^{p_n-1} \int_{0}^{1} |f(s,U^*(\xi))-g(U^*(\xi))|~|\phi(sT_n+qT_n,\xi)- \phi(qT_n,\xi)|~ds~d\xi. &&
\end{flalign*}
The quantity $|f(s,U^*(\xi))-g(U^*(\xi))|$ is bounded by a constant $M^*$. Consequently
\begin{flalign*}
|J^3_n| & \leq (\xi_1-\xi_0) T_n(p_n-k_n)M^*\e \leq (\xi_1-\xi_0)(b-a) M^*\epsilon.
\end{flalign*}
So, we proved that $\lim\limits_{n\to +\infty} J^3_n=0$. Consequently, we have $\lim\limits_{n\to +\infty} J^1_n=0$, which concludes the proof.
\end{proof}
\begin{proposition}
The function $U^*$ is of class $\mathcal{C}^{1,2}(\R)$ and
\begin{equation*}
(U^*)_t-c^*(U^*)_{\xi}-(U^*)_{\xi \xi}-g(U^*)=0~~\text{on}~\R^2.
\end{equation*}
\end{proposition}
\begin{proof}
We know that $U_{T_n}$ is a function of class $\mathcal{C}^{1,2}(\R^2)$ satisfying the equation
\begin{equation}
\label{one}
(U_{T_n})_t-c_n(U_{T_n})_{\xi}-(U_{T_n})_{\xi \xi}-f^{T_n}(t,U_{T_n})=0,~~\text{on}~\R^2. 
\end{equation}
We also know that $(U_{T_n})_n$ converges up to extraction of a subsequence to $U^*$ in $W^{1,2;p}_{\text{loc}}(\R^2)$, for any $1 < p <+\infty$. In particular, the convergence takes place for $p=2$. Let $\phi \in \mathcal{C}_c^\infty(\R^2)$. Multiplying Equation \eqref{one} by $\phi$, and integrating by parts on $\R^2$, we obtain
\begin{equation}
\label{two}
-\int_{\R^2}U_{T_n}\phi_t-c_n\int_{\R^2}(U_{T_n})_\xi\phi+\int_{\R^2}(U_{T_n})_\xi\phi_\xi-\int_{\R^2}f^{T_n}(t,U_{T_n})\phi=0.
\end{equation}
We have
\begin{flalign*}
\big|\int_{\R^2}U_{T_n}\phi_t-\int_{\R^2}U^*\phi_t\big|& \leq \|U_{T_n}-U^*\|_{L^2(\text{Supp}(\phi))}\|\phi_t\|_{L^2(\text{Supp}(\phi))}&&\\
&\leq \|U_{T_n}-U^*\|_{W^{1,2,2}(\text{Supp}(\phi))}\|\phi_t\|_{L^2(\text{Supp}(\phi))}. 
\end{flalign*}
Also 
$$\int_{\R^2}(U_{T_n})_\xi\phi \xrightarrow{n \to \infty} \int_{\R^2}(U^*)_\xi\phi~~\text{and}~~\int_{\R^2}(U_{T_n})_\xi\phi_\xi \xrightarrow{n \to \infty} \int_{\R^2}(U^*)_\xi\phi_\xi.$$
Furthermore, we showed in Lemma \ref{distribution} that
$$\int_{\R^2}f(\frac{t}{T_n},U_{T_n})\phi \xrightarrow{n \to \infty} \int_{\R^2}g(U^*)\phi.$$
Consequently, if we pass to the limit as $n\to +\infty$ in \eqref{two}, we have that
\begin{equation*}
\label{three}
-\int_{\R^2}U^*\phi_t-c^*\int_{\R^2}(U^*)_\xi\phi+\int_{\R^2}(U^*)_\xi\phi_\xi-\int_{\R^2}g(U^*)\phi=0.
\end{equation*}
Consequently $U^*$ is a weak solution of the equation
\begin{equation*}
(U^*)_t-c^*(U^*)_{\xi}-(U^*)_{\xi \xi}-g(U^*)=0~~\text{on}~\R^2.
\end{equation*}
By parabolic regularity theory, the function $U^*$ is in fact a function $\mathcal{C}^{1,2}(\R^2)$, and it satisfies the previous equation in the classical sense on $\R^2$.
\end{proof}
\begin{proposition}
The function $U^*$ does not depend on $t$. In other terms
$$\partial_t{U^*}(t,\xi)=0,~~~~\forall t \in \R,~~\forall \xi \in \R.$$
\end{proposition}
\begin{proof}
Let $(t,\xi) \in \R^2$ and $n \in \N$. Let $K$ be a compact set containing $(t,\xi)$ and $(0,\xi)$. According to the previous proposition, there exists $C_K \in \R$ such that $\|U_{T_n}\|_{\mathcal{C}^{1,2}(K)}\leq C_K$. Let  $k_n\in \Z$ such that $|t-k_nT_n|\leq T_n$. We have
$$|U^*(t,\xi)-U^*(0,\xi)| \leq|U^*(t,\xi)-U_{T_n}(t,\xi)|+|U_{T_n}(t,\xi)-U_{T_n}(0,\xi)|+|U_{T_n}(0,\xi)-U^*(0,\xi)|.$$
The first and the third term of the previous inequality are smaller than $\|U^*-U_{T_n}\|_{\infty,K}$. Let us examine the second term of the sum above:   
$$|U_{T_n}(t,\xi)-U_{T_n}(0,\xi)|=|U_{T_n}(t-k_nT_n,\xi)-U_{T_n}(0,\xi)| \leq \|U_{T_n}\|_{\mathcal{C}^{1,2}(K)}|t-k_nT_n|\leq C_K T_n.$$
Passing to the limit when $n$ tends to infinity, it occurs that
$$U^*(t,\xi)=U^*(0,\xi).$$
That is $U^*$ do not depend on $t$.
\end{proof}
\begin{proposition}
We have $c^*=c_g$ and $U^*\equiv U_g$.
\end{proposition}
\begin{proof}
For all $n\geq0$, we have $U_{T_n}(0,0)=\theta_g$. This thus implies that $U^*(0)=\theta_g$.
The couple $(c^*,U^*)$ satisfies in the classical sense the equation $(U^*)''+c^*(U^*)'+g(U^*)=0$ on $\R$.
We are now going to justify that $U^*(-\infty)=1$ and that $U^*(+\infty)=0$. Knowing that $U^*(0)=\theta_g$, it could be  possible a priori that $U^*\equiv \theta_g$. We are however going to show bwoc that this situation cannot occur. We thus suppose that $U^*\equiv \theta_g$. As $(U_{T_n})_\xi<0$ on $\R$, there exists a unique positive real number $\xi_n$ such that $\textstyle{U_{T_n}(0,\xi_n)=\frac{\theta_g}{2}}$.
We thus define the function $V_n:\R^2\to\R$ by $V_n(t,\xi)=U_{T_n}(t,\xi+\xi_n)$. Consequently, we have
\begin{equation}
\label{eqfin}
\begin{cases} V_n(0,0)=\frac{\theta_g}{2}, \\ V_n(t,\xi)\leq U_{T_n}(t,\xi)~\text{on}~\R^2. \end{cases}
\end{equation}
The sequence $(V_n)_n$ converges up to extraction of a subsequence in $W^{1,2;p}_{\text{loc}}(\R^2)$ weakly and in $\mathcal{C}^{0,\alpha}_{\text{loc}}(\R^2)$ for any $1< p<+\infty$ and for any $\alpha \in (0,1)$ to a function $V^*\in \mathcal{C}^2(\R)$  satisfying
\begin{equation}
\label{signe}
(V^*)''+c^*(V^*)'+g(V^*)=0~~\text{on}~\R.
\end{equation}
If we pass to the limit as $n$ tends to the infinity in \eqref{eqfin}, we obtain that $$\begin{cases} V^*(0)=\frac{\theta_g}{2}, \\ V^*\leq \theta_g~\text{on}~\R. \end{cases}$$
Consequently, we have that $V^*(-\infty)=\theta_g~\text{and}~V^*(+\infty)=0$.
We multiply \eqref{signe} by $(V^*)'$, then we integrate on $\R$. It occurs that
$$c^*\int_{-\infty}^{+\infty}[(V^*)'(\xi)]^2d\xi=-\int_{-\infty}^{+\infty}g[(V^*)(\xi)](V^*)'(\xi) d\xi=\int_0^{\theta_g}g(s)ds<0.$$
So, $$c^*<0.$$
The same reasoning on the function $W_n$ defined by $W_n(t,\xi)=U_{T_n}(t,\xi+\tilde{\xi}_n)$ such that $\textstyle{W_n(0,0)=\frac{1+\theta_g}{2}}$ leads to the fact that $c^*$ has the same sign as $\textstyle{\int_{\theta_g} ^1 g(s) ds}$, that is $c^*>0$. We get to a contradiction. The situation $U^*\equiv \theta_g$ can not occur. Consequently,  $U^*$ is not identically equal in $\theta_g$. Applying the strong maximum principle, one gets
$$(U^*)'<0~~\text{on}~\R.$$
As $U^*(0)=\theta_g$, we have necessarily $U^*(-\infty)=1~\text{and}~U^*(+\infty)=0$. To summarize, the couple $(c^*,U^*)$ satisfies
$$\begin{cases}
(U^*)''+c^*(U^*)'+g(U^*)=0~~~~~\text{on}~\R,\\
U^*(-\infty)=1,~~U^*(+\infty)=0,\\
U^*(0)=\theta_g.
\end{cases}$$
Knowing that this problem admits a unique solution, il occurs that $c^*=c_g~\text{and}~U^*\equiv U_g$.
The uniqueness of accumulation point of $(c_{T_n})_n$ and $(U_{T_n})_n$ imply that $c_T\xrightarrow{T\to0}c_g$ and $U_T\xrightarrow{T\to0} U_g$ in $W^{1,2;p}_{\text{loc}}(\R^2)$ weakly and in $\mathcal{C}^{0,\alpha}_{\text{loc}}(\R^2)$ for any $1< p<+\infty$ and for any $\alpha \in (0,1)$. 
\end{proof}
\section{Pulsating fronts for small perturbations of the nonlinearity}
\subsection{Existence of  pulsating fronts}
This section is devoted to the proof of Theorems \ref{pertex} and \ref{pertconv}. We remind the reader that in this section, the $\text{Poincar\'e}$ map associated with the function $f^T$ has exactly two stable fixed points $0$ and $1$, and a unique  unstable fixed point $\alpha_0$ between both. According to \cite{ABC}, there exists a unique pulsating front $(c_T,U_T)$ solving \eqref{travel} with $U_T(0,0)=\alpha_0$. We call $w_0(t)$ the solution of the equation $y'=f^T(t,y)$ satisfying $y(0)=\alpha_0$. We saw in Section \ref{characterization} that $w_0$ is a $T$-periodic function. Furthermore, we have 
$$\lambda_{w_0,f^T}=-\frac{1}{T}\int_0^Tf_u(s,w_0(s))ds<0.$$
We give a corollary of the Gr\"onwall lemma.
\begin{lemma}
\label{gronwall}
Let $y:[a,b] \to \R$ a function of class $\mathcal{C}^1$ such that\\
$$\exists \alpha>0, \quad \exists \beta>0, \quad \forall t \in [a,b],\qquad |y'(t)| \leq \beta + \alpha|y(t)|.$$
So, we have
$$\forall t \in [a,b], \qquad |y(t)| \leq |y(a)|e^{\alpha(t-a)}+\frac{\beta}{\alpha}(e^{\alpha(t-a)}-1).$$ 
\end{lemma}
\begin{proof}
For all $t$ in $[a,b]$, we have\\
$|y(t)| \leq |y(a)|+|y(t)-y(a)| \leq  |y(a)|+\int_a^t |y'(s)|ds \leq  |y(a)|+\beta(t-a)+\int_a^t |y(s)|ds$.\\
We then apply the Gr\"onwall lemma.
\end{proof}
\noindent
We denote by $P_\e$ (resp. $P$) the $\text{Poincar\'e}$ map associated with $f^{T,\e}$ (resp. $f^T$). We are going to show here that $P_\e$ (resp. $P_\e')$ converges uniformly on [0,1] to $P$ (resp. $P'$) when $\e$ tends to $0$.
Let us specify that according to \eqref{01}, \eqref{benjamin}, \eqref{perturbation2} and the mean value theorem, we have 
\begin{equation}
\label{perturbation1}
|f^T(t,u)-f^{T,\e}(t,u)|\leq\omega(\e),~~\forall (t,u) \in [0,T] \times [0,1].
\end{equation}
\begin{proposition}
\label{label}
There exists three constants $C_1,~C_2$ and $C_3$ such that for any $\e>0$, we have\\
\begin{equation}
\label{Poincare}
\|P-P_\e\|_{\infty,[0,1]}<C_1\omega(\e).
\end{equation}
\begin{equation}
\label{Poincareder}
\|P'-P_\e'\|_{\infty,[0,1]}<C_2 \sinh(C_3 \omega(\e)).
\end{equation}
\end{proposition}
\begin{proof}
We show the first relation. Let $\alpha \in [0,1]$ and $\e>0$. We define $u_\e=w(\alpha,\cdot)-w_\e(\alpha,\cdot)$. There exist two positive constants $C$ and $C'$ such that for any $t \in [0,T]$, we have 
\begin{align}
|u_\e '(t)|  & = |f^T(t,w(\alpha ,t))-f^{T,\e}(t,w_\e (\alpha ,t))| \nonumber\\
& \leq |f^T(t,w(\alpha ,t))-f^T(t,w_\e (\alpha ,t))|+|f^T(t,w_\e (\alpha ,t))-f^{T,\e}(t,w_\e (\alpha ,t))| \nonumber\\
 & \leq C|u_\e(t)|+C'\omega(\e)\nonumber ~~(\text{by}~\eqref{perturbation1}).
\end{align}
We apply Lemma \ref{gronwall}. For all $t\in[0,T]$, we have
$|u_\e(t)|\leq \frac{C'\omega(\e)}{C}(e^{Ct}-1)$.\\ 
Consequently, if we take $t=T$, and if we define $C_1= \frac{C'}{C}(e^{CT}-1)$, then we conclude 
$$|P(\alpha)-P_\e(\alpha)|<C_1\omega(\e).$$
We show the second relation. Let $\alpha \in [0,1]$ and $\e>0$.
$$P'(\alpha)-P_\e'(\alpha) = e^{ \int_0^T f^T_u(s,w(\alpha,s)) ds}-e^{\int_0^T f^{T,\e}_u(s,w_\e(\alpha,s)) ds }$$
$$=   \underbrace{e^{\frac{1}{2}\big( \int_0^T f^T_u(s,w(\alpha,s)) ds + \int_0^T f^{T,\e}_u(s,w_\e(\alpha,s)) ds \big)}}_{(I)} \times 2\sinh\big(\frac{1}{2}\underbrace{\int_0^T f^T_u(s,w(\alpha,s))-f^{T,\e}_u(s,w_\e(\alpha,s))ds}_{(II)}\big).$$
Knowing that $w_\e(\alpha,\cdot)\in [0,1]$ and that $f_u^{T,\e}$ is bounded independently of $\e$ on $[0,T] \times [0,1]$ (because $w_\e$ is bounded), the term $(I)$ can be bounded by a constant which we shall denote $C_2$. Let us now consider the term $(II)$\\
$(II)  = \int_0^T f^T_u(s,w(\alpha,s))-f^T_u(s,w_\e(\alpha,s)) ds +\int_0^T  f^T_u(s,w_\e(\alpha,s))-f^{T,\e}_u(s,w_\e(\alpha,s)) ds.$ 
\vspace{0.5em}\\
According to \eqref{Poincare} and the mean value theorem, there exists $M>0$ such that\\
$$f^T_u(s,w(\alpha,s))-f^T_u(s,w_\e(\alpha,s)) \leq  M |w(\alpha,s)-w_\e(\alpha,s)| \leq M C_1\omega(\e).$$
Furthermore, according to \eqref{perturbation2}, we have\\
$$f^T_u(s,w_\e(\alpha,s))-f^{T,\e}_u(s,w_\e(\alpha,s))\leq \omega(\e).$$
Consequently, $$(II) \leq (M C_1+1)T\omega(\e).$$ \\
We define $C_3=(M C_1+1)T$. Finally, we have
$$P'(\alpha)-P'_\e(\alpha) \leq C_2 \sinh(C_3\omega(\e)),$$
and the proof of Proposition \ref{label} is complete.
\end{proof}
\begin{proposition}
There exists $\e_0>0$ small enough such that for all $\e\in(0,\e_0)$, $P_\e$ has exactly two stable fixed points $0$ et $1$ and a unique unstable fixed point $\alpha_0^\e \in (0,1)$.
\end{proposition}
\begin{proof}
We begin by making some remarks on $P$. As $P'(0)<1$ (resp. $P'(1)<1$), by continuity there exists an interval $[0,x_0]$ (resp. $[x_1,1]$) on which $P'<1$ (resp. $P'<1$).  As $P'(\alpha_0)>1$, there exists an interval $[\alpha_0-\eta,\alpha_0+\eta]$ on which $P'>1$. 
\vspace{0.5em}\\
The properties of $P$ above can be extend to $P_\e$. Indeed, according to \eqref{Poincare} and \eqref{Poincareder}, there exists $\e_0>0$ such that for any $\e \in(0,\e_0)$, $P_\e'<1$  on $[0,x_0]$, $P_\e'<1$  on $[x_1,1]$, and $P_\e'>1$ on $[\alpha_0-\eta,\alpha_0+\eta]$. Consequently, for $\e \in (0,\e_0)$, as $P_\e(0)=P_\e(1)=0$, the function $P_\e$ has necessarily another fixed point $\alpha_0^\e$ between both.
\vspace{0.5em}\\
Let $\e\in(0,\e_0)$. The function $P_\e$ has no fixed point on $[0,x_0]$ and on $[x_1,1]$. We make the proof for $[0,x_0]$. Let $\alpha \in [0,x_0]$. According to Taylor-Lagrange's formula, we have $P_\e(\alpha)=P_\e(0)+\alpha P_\e'(\hat{\alpha})$, where $\hat{\alpha} \in [0,\alpha] \subset  [0,x_0]$. According to the fact that $P_\e(0)=0$ and $P_\e'(\hat{\alpha})<1$, we have $P_\e(\alpha)<\alpha$.
\vspace{0.5em}\\
The function $P_\e$ has no fixed point on $[x_0,\alpha_0-\eta]$ and on $[\alpha_0+\eta,x_1]$. Indeed, we can find $\nu_0$ such that for any $ \alpha \in [x_0,\alpha_0-\eta]$, we have $P(\alpha)-\alpha>\nu_0$. Consequently, according to \eqref{Poincare}, even if it means setting $\e_0$ smaller, we can prove that for any $ \alpha$ in $[x_0,\alpha_0-\eta]$, we have $\textstyle{P_\e(\alpha)-\alpha>\frac{\nu_0}{2}}$.
\vspace{0.5em}\\
Finally, $P_\e$ has a unique fixed point (on $[\alpha_0-\eta,\alpha_0+\eta]$). Indeed, if there was others, we would have the existence of a point $\alpha_1^\e$ for which $P_\e'(\alpha_1^\e)<1$. It is contradictory to the fact that $P_\e'>1$ on $[\alpha_0-\eta,\alpha_0+\eta]$.
\end{proof}
\noindent
Consequently, if $\e \in (0,\e_0)$, then we are under the hypotheses of the existence and uniqueness theorem of Alikakos, Bates and Chen \cite{ABC}. In particular, there exists a pulsating front $(U_\e,c_\e)$ solving \eqref{travelepsilon}.
\subsection{Convergence of the couple $(c_\e,U_\e)$ as $\e \to 0$.}
This subsection is devoted to the proof of Theorem \ref{pertconv}. By Theorem \ref{pertex}, for $\e \in (0,\e_0)$, there exists a pulsating front $(U_\e,c_\e)$ solving \eqref{travelepsilon}. As in Section \ref{sectionpetite}, we can show there exists a couple $(c^*,U^*)$ such that as $\e \to 0$, $c_\e$ converges to $c^*$ and $U_\e$ converges to $U^*$ in $W^{1,2;p}_{\text{loc}}(\R^2)$ weakly and in $\mathcal{C}^{0,\alpha}_{\text{loc}}(\R^2)$ for any $1< p < +\infty$ and any $\alpha \in (0,1)$. We prove then that  $(c^*,U^*)$ is solution, at first in the sense of distributions, but also  in the classical sense of the equation 
$$(U^*)_t-c^*(U^*)_\xi-(U^*)_{\xi \xi} = f^T(t,U^*),~~\text{on}~\R^2.$$
We also have that $U^*(T,\cdot)=U^*(0,\cdot)~\text{on}~\R,$ and $U^*(0,0)=\alpha_0.$ Consequently, if we prove that $U^*(\cdot,-\infty)=1~~\text{and}~~U^*(\cdot,+\infty)=0~~\text{uniformly on}~\R$, 
then the couple $(c^*,U^*)$ solves Problem \eqref{travel} with $U^*(0,0)=\alpha_0$. By uniqueness, we shall have $c^*=c_T~\text{and}~U^*=U_T.$
We can not prove this result in the same way as in Section \ref{sectionpetite} because here, we do not know the sign of $c_T$. We use a technique which comes from Lemma $6.5$ of \cite{BH02} where they are interested in the exponential behavior of the front at infinity. After, we shall denote 
$$U_\pm^*(t)=\lim\limits_{\xi \to \pm \infty} U^*(t,\xi).$$
Knowing that $U^*(0,0)=w_0(0)$ and $\partial_\xi U^*\leq 0$, we have a priori that $U_+^*\equiv w_0~\text{or}~U_+^*\equiv 0$. We suppose at first that $U_+^*\equiv w_0$. So, we have that $U^*\geq w_0$. Since the two functions are equal on $(0,0)$, the strong maximum principle implies that 
$$U^*(t,\xi) = w_0(t),~~~\forall t\leq 0,~~\forall \xi \in \R.$$
By periodicity, the equality is also true when $t$ is positive. Consequently $U^* \equiv w_0$ on $\R^2$. As $\partial_\xi U_\e<0$, there exists a unique positive real number $\xi_\e$ such that $\textstyle{U_\e(0,\xi_\e)=\frac{\alpha_0}{2}}$. We define the function $V_\e:\R^2\to\R$ by $V_\e(t,\xi):=U_\e(t,\xi+\xi_\e).$ We have so 
$$\begin{cases} V_\e(0,0)=\frac{\alpha_0}{2}, \\ V_\e \leq U_\e~~\text{on}~\R^2, \\ \partial_\xi V_\e <0~~\text{on}~\R^2.\end{cases}$$
We can show that $V_\e$ converges to a function $V^*$ in $W_{loc}^{1,2;p}(\R^2)$ weakly and in $\mathcal{C}^{0,\alpha}_{loc}(\R^2)$ for $1<p\leq+\infty$ and $\alpha \in (0,1)$. Actually, $V^*$ is in $\mathcal{C}^{1,2}(\R^2)$ and satisfies in the classical sense $\partial_tV^*-c^* \partial_\xi V^*-\partial_{\xi \xi}V^*=f^T(t,V^*)~\text{on}~\R^2$.
Furthermore, we have 
$$\begin{cases} V^*(0,0)=\frac{\alpha_0}{2}, \\ V^* \leq U^* \equiv w_0~~\text{on}~\R^2, \\ \partial_\xi V^* < 0~~\text{on}~\R^2.\end{cases}$$
The last inequality is a consequence of the strong maximum principle. We define the function $Z:\R^2 \to \R$ by $Z(t,\xi)=w_0(t)-V^*(t,\xi)$. The function $Z$ is positive and satisfies 
$$\partial_t Z-c^*\partial_\xi Z-\partial_{\xi\xi} Z=\alpha(t,\xi)Z~\text{on}~\R^2,$$
where $\alpha(t,\xi)=\begin{cases}\frac{f^T(t,w_0(t))-f^T(t,V^*(t,\xi))}{w_0(t)-V^*(t,\xi)}~~\text{if}~ w_0(t)\neq V^*(t,\xi), \\  f^T_u(t,w_0(t)) \hspace{8.2em} \text{otherwise}. \end{cases}$ 
\vspace{0.5em}\\
We also have  
$$Z(t+T,\xi)=Z(t,\xi),~~~~\forall t \in \R,~~\forall \xi \in \R.$$
We are going to show that $\textstyle{\frac{Z_\xi}{Z}}$ is bounded on $\R^2$. Let $(t,\xi) \in \R^2$. Standard parabolic estimates imply that there exists a constant $C_1$ independent of $(t,\xi)$ (because the function $c$ is bounded) such that 
\begin{equation}
\label{PAR}
0 \leq Z_\xi(t,\xi) \leq C_1 \max_{[t-1,t]\times[\xi-1,\xi+1]}Z.
\end{equation}
Furthermore, the parabolic Harnack's inequality implies that there exists a constant $C_2$ independent of $(t,\xi)$ (because the function $c$ is bounded) such that 
\begin{equation}
\label{HAR}
\max_{[t-1,t]\times[\xi-1,\xi+1]}Z \leq C_2Z(t+T,\xi)=C_2Z(t,\xi).
\end{equation}
Inequalities \eqref{PAR} and \eqref{HAR} imply that 
$$0<\frac{Z_\xi}{Z}(t,\xi)<C_1C_2.$$
Consequently $\limsup_{\xi \to-\infty ,~~ t\in \R}\frac{Z_\xi}{Z}(t,\xi)$ is a finite and non negative real number called $\beta$. So, there exists two sequences $(t_n)_n$ and $(\xi_n) _n$ with $\xi_n \xrightarrow{n \to +\infty} -\infty$ such that 
\begin{equation}
\label{SUITE}
\lim\limits_{n\to+\infty}\frac{Z_\xi}{Z}(t_n,\xi_n)=\beta.
\end{equation}
Let us note that we can show in the same way that the quotient $\frac{Z_t}{Z}$ is bounded in $\R^2$.\\
Let $n\in \N$. We write $t_n=k_nT+t_n'$, where $k_n$ is an integer and $t'_n$ satisfies $|t_n'|<T$. We define the function $h_n:\R^2\to\R$ by
$$h_n(t,\xi):=\frac{Z(t,\xi+\xi_n)}{Z(t_n,\xi_n)}.$$
Let us show that the function $h_n$ is locally bounded. Let $K$ be a compact set of $\R^2$, and $(t,\xi)\in K$.
\begin{flalign*}
|\ln(h_n(t,\xi))| & =|\ln(Z(t,\xi+\xi_n))-\ln(Z(t_n,\xi_n))| \\
& =|\ln(Z(t+k_nT,\xi+\xi_n))-\ln(Z(t_n,\xi_n))| \\
& \leq  \|(\log(Z))_t\|_\infty |t-t_n'| + \|(\log(Z))_\xi\|_\infty |\xi| .
\end{flalign*}
Let us specify that $\|(\log(Z))_t\|_\infty$ and $ \|(\log(Z))_\xi\|_\infty$ are well defined because the two quotients $\textstyle{\frac{Z_\xi}{Z}}$ and $\textstyle{\frac{Z_t}{Z}}$ are bounded on $\R^2$. Consequently, the quantity above is bounded because $(t,\xi)\in K$ and $|t_n'|<T$. We have 
\begin{equation}
\label{hn}
(h_n)_t-c^*(h_n)_\xi-(h_n)_{\xi\xi}=\alpha_n(t,\xi)h_n~~\text{on}~\R^2,
\end{equation}
\vspace{-0.7em}\\
where $\alpha_n(t,\xi)=\frac{f^T(t,w_0(t))-f^T(t,V^*(t,\xi+\xi_n))}{w_0(t)-V^*(t,\xi+\xi_n)}$.
\vspace{0.2em}\\
The sequence $(\alpha_n)_n$ is bounded by $\textstyle \max_{(t,u) \in [0,T] \times [0,1]}{|f^T_u|}$ in $L^\infty_{loc}(\R^2)$ independently of $n$. Furthermore $(h_n)_n$ is locally bounded. Consequently, we can show that $(h_n)_n$ converges to a function $h$ in $W^{1,2;p}_{\text{loc}}(\R^2)$ weakly and in $\mathcal{C}^{0,\alpha}_{\text{loc}}(\R^2)$ for any $1 < p <+\infty$ and for any $\alpha \in (0,1)$. Actually, passing to the limit in the equation \eqref{hn}, $h$ satisfies (at first in the sense of distributions, but in fact in the classical sense using standard parabolic estimates) 
\begin{equation}
\label{h}
h_t-c^*h_\xi-h_{\xi\xi}=f_u^T(t,w_0(t))h~~\text{on}~\R^2.
\end{equation}
We have 
$$\frac{(h_n)_\xi}{h_n}(t,\xi)=\frac{Z_\xi}{Z}(t,\xi+\xi_n),~~\forall t\in \R,~~\forall \xi\in \R.$$
As $\xi_n \xrightarrow{n \to \infty}0$, we have that 
$$\lim\limits_{n \to +\infty}\frac{Z_\xi}{Z}(t,\xi+\xi_n) \leq \limsup\limits_{\xi \to -\infty~;~t\in \R}\frac{Z_\xi}{Z}(t,\xi):=\beta,~~\forall t\in \R,~~\forall \xi\in \R.$$
So
$$\lim\limits_{n \to \infty} \frac{(h_n)_\xi}{h_n}(t,\xi) \leq \beta,~~\forall t\in \R,~~\forall \xi\in \R.$$
We also have$$\lim\limits_{n \to \infty} \frac{(h_n)_\xi}{h_n}(t,\xi)=\frac{h_\xi}{h}(t,\xi),~~\forall t\in \R,~~\forall \xi\in \R.$$
Consequently
\begin{equation*}
\frac{h_\xi}{h}(t,\xi) \leq \beta,~~\forall t\in \R,~~\forall \xi\in \R.
\end{equation*}
Furthermore, first we have that
$$\frac{(h_n)_\xi}{h_n}(t_n',0)=\frac{Z_\xi}{Z}(t_n,\xi_n) \xrightarrow{n \to +\infty} \beta,$$
and second, as $(t_n')_n$ is bounded, it  converges up to extraction of a subsequence to a constant $t_0\in \R$. Consequently
\begin{equation*}
\frac{h_\xi}{h}(t_0,0) = \beta.
\end{equation*}
The function $\frac{h_\xi}{h}$ satisfies on $\R^2$ the equation
$$(\frac{h_\xi}{h})_t-c^*(\frac{h_\xi}{h})_\xi-(\frac{h_\xi}{h})_{\xi\xi}-2(\frac{h_\xi}{h})_\xi(\frac{h_\xi}{h})=0.$$
So, applying the maximum principle and using the $T$-periodicity, we obtain that
$$h_\xi(t,\xi)=\beta h(t,\xi),~~~~\forall t \in \R,~~\forall \xi \in \R.$$
Consequently, there exists a $T$-periodic positive function $\Gamma(t)$ such that
$$h(t,\xi)=e^{\beta \xi}\Gamma(t),~~~~\forall t \in \R,~~\forall \xi \in \R.$$
We put back the previous expression in \eqref{h}, then we simplify by $e^{\beta \xi}$. We obtain
$$\Gamma'(t)-\beta c^* \Gamma(t)- \beta^2 \Gamma(t) =f_u^T(t,w_0(t))\Gamma(t),~~~~\forall t \in \R.$$
We divide then by $\Gamma(t)$ and we integrate between $0$ and $T$
$$\int_0^T \frac{\Gamma'(t)}{\Gamma(t)}dt - \beta T c^* -\beta^2 T = \int_0^T f_u^T(t,w_0(t))dt.$$
According to the fact that $\textstyle{\int_0^T \frac{\Gamma'(t)}{\Gamma(t)}dt=0}$, we have
$$\beta^2 + c^*\beta - \lambda_{w_0,f^T}=0.$$
Consequently, $\beta$ is a a non negative root of $X^2+c^*X+\lambda_{w_0,f^T}=0$.
\vspace{1em}\\
As $\partial_\xi U_\e<0$, there exists a unique negative $\overline{\xi}_\e$ such that $U_\e(0,\overline{\xi}_\e)=\frac{1+\alpha_0}{2}$. We define the function $W_\e:\R^2\to\R$ by $W_\e(t,\xi):=U_\e(t,x+\overline{\xi}_\e)$
We have
$$\begin{cases} W_\e(0,0)=\frac{\alpha_0+1}{2}, \\ W_\e \geq U_\e~~\text{on}~\R^2, \\ \partial_\xi W_\e <0~~\text{on}~\R^2.\end{cases}$$
The function $W_\e$ converges to a function $W^*$ in $W_{loc}^{1,2;p}(\R^2)$ weakly and in $\mathcal{C}^{0,\alpha}_{\text{loc}}(\R^2)$ for $1<p<+\infty$ and $\alpha \in (0,1)$. Furthermore, $W^*$ is in $\mathcal{C}^{1,2}(\R^2)$ and satisfies $\partial_tW^*-c^* \partial_\xi W^*-\partial_{\xi \xi}W^*=f^T(t,W^*)~\text{on}~\R^2.$
Consequently, we have
$$\begin{cases} W^*(0,0)=\frac{\alpha_0+1}{2}, \\ W^* \geq U^* \equiv w_0~~\text{on}~\R^2, \\ \partial_\xi W^* < 0~~\text{on}~\R^2.\end{cases}$$
The fact that the last inequality is strict is a consequence of the strong maximum principle. We define $\overline{Z}(t,\xi)=w_0(t)-W^*(t,\xi)$. We show using parabolic standards estimates and the parabolic Harnack's inequality that the the negative function $\textstyle{\frac{\overline{Z}_\xi}{\overline{Z}}}$ is bounded on $\R^2$. Consequently $\textstyle{\limsup\limits_{\xi \to-\infty ~;~ t\in \R}\frac{\overline{Z}_\xi}{\overline{Z}}(t,\xi)}$ is a finite and a non positive real called $\delta$. As previously, we also demonstrate that 
$\delta$ is a non positive root of $X^2+c^*X-\lambda_{w_0,f^T}$. Finally, the product $\beta \delta$ is a non positive real number, but  it is equal to $-\lambda_{w_0,f^T}$ which is by hypothesis a positive real number.
So, we have a contradiction, and the hypothetis $U^*_+ \equiv w_0$ is not valid. So we have that $U_+^* \equiv 0.$ Consequently, we have $U_-^* \equiv 1$. Indeed, If it was not the case, we would have $U_-^* \equiv w_0$. So we would have
$$U^*(t,\xi) \leq w_0(t),~~~~\forall t \in \R,~~\forall \xi \in \R .$$
Yet, the two functions are equal on $(0,0)$, so, by the strong maximum principle we would have
$$U^*(t,\xi)=w_0(t),~~~~\forall t\leq0,~~ \forall \xi \in \R .$$
By periodicity, the equality would be true for all $t\in\R$. This would contradict the fact that $U_+^* \equiv 0$.
\vspace{1em}\\
So we have 
$$U^*(\cdot,-\infty)=1~~\text{and}~~U^*(\cdot,+\infty)=0,$$
and the proof of Theorem \ref{pertconv} is complete.

\section*{Acknowledgements}
This work has been carried out in the framework of Archimedes LabEx (ANR-11-LABX-0033) and of the A*MIDEX project (ANR-11-IDEX-0001-02), funded by the ''Investissements d'Avenir'' French Government program managed by the French National Research Agency (ANR). The research leading to these results has received funding from the European Research Council under the European Union's Seventh Framework Programme (FP/2007-2013) / ERC Grant Agreement n.321186 - ReaDi - Reaction-Diffusion Equations, Propagation and Modelling, and from the ANR project NONLOCAL (ANR-14-CE25-0013).

\bibliographystyle{plain}
\bibliography{Biblio1}

\end{document}